\documentclass[smallextended,numbook,runningheads]{svjour3}
\usepackage{amsmath}
\smartqed  

\usepackage{amsfonts,amssymb}
\usepackage{epsfig}
\usepackage{booktabs}
\usepackage{xcolor}

\journalname{}
\date{ \phantom{b} \vspace{45mm}\phantom{e}}
\def\makeheadbox{\hspace{-4mm}Version of  16 July 2019}

\def\bigo{{\mathcal O}}
\def\real{{\mathbf R}}
\def\complex{{\mathbf C}}

\def\integer{{\mathbf Z}}

\def\bfz{{\mathbf z}}
\def\bfzeta{{\boldsymbol{\zeta}}}

\def\d{{\mathrm d}}
\def\e{{\mathrm e}}

\def\iu{\mathrm{i}}
\def\eps{\varepsilon}

\DeclareMathOperator{\sinc}{sinc}
\DeclareMathOperator{\tanc}{tanc}
\DeclareMathOperator{\sinch}{sinch}
\DeclareMathOperator{\tanch}{tanch}

\def\Re{{\mathrm{Re}\,}}
\def\Im{{\mathrm{Im}\,}}

\newdimen\GGGlength
\newdimen\GGGheight
\newbox\GGGbox
\def\GGGput[#1,#2](#3,#4)#5{%
  \setbox\GGGbox\vbox{\hbox{#5}\kern0pt}%
  \GGGlength\wd\GGGbox%
  \divide\GGGlength by100 \multiply\GGGlength by#1%
  \GGGheight\ht\GGGbox%
  \divide\GGGheight by100 \multiply\GGGheight by#2%
  \put(#3,#4){\kern-\GGGlength\raise-\GGGheight\box\GGGbox}}

  \newtheorem{algorithm}[theorem]{Algorithm}

\begin{document}

\title{A filtered Boris algorithm for charged-particle dynamics in a strong
magnetic field}

\titlerunning{A filtered Boris algorithm for charged-particle dynamics}

\author{Ernst Hairer$^1$, Christian Lubich$^2$, Bin~Wang$^3$}
\authorrunning{E.\ Hairer, Ch.\ Lubich, B.\ Wang}

\institute{$^1$~Dept.\ de Math{\'e}matiques, Univ.\ de Gen{\`e}ve,
CH-1211 Gen{\`e}ve 24, Switzerland,\\
\phantom{$^1$~}\email{Ernst.Hairer@unige.ch}\\
$^2$~Mathematisches Institut, Univ.\ T\"ubingen, D-72076 T\"ubingen, Germany,\\
\phantom{$^2$~}\email{Lubich@na.uni-tuebingen.de}\\
$^3$~School of Mathematical Sciences, Qufu Normal University, Qufu 273165, P.R.China, \\
\phantom{$^3$~}\email{wangbinmaths@qq.com}
}

\date{ }

\maketitle

\begin{abstract}
A modification of the standard Boris algorithm, called {\it filtered} Boris algorithm, is proposed for the numerical integration of the equations of motion of charged particles in a strong non-uniform magnetic field in the asymptotic scaling known as maximal ordering. With an appropriate choice of filters, second-order error bounds in the position and in the parallel velocity, and first-order error bounds in the normal velocity  are obtained with respect to the scaling parameter.  The proof compares the modulated Fourier expansions of the exact and the numerical solutions. Numerical experiments illustrate the error behaviour of the filtered Boris algorithm.\bigskip

\bigskip

\end{abstract}

\section{Introduction}\label{sec:intro}

In this paper we propose and analyse a numerical integrator for the equations of motion of a charged particle in a strong inhomogeneous magnetic field,
\begin{equation}\label{ode}
\begin{aligned}
&\ddot x(t) =  \dot x(t) \times B(x(t),t) + E(x(t),t)
\\
&\text{with}\quad
B(x,t) = \frac 1 \eps \, B_0(\eps x  ) + B_1(x,t) \quad\text{for }\  0<\eps\ll 1 .
\end{aligned}
\end{equation}
This scaling is of interest in particle methods in plasma physics and is called {\it maximal ordering} in \cite{brizard07fon}; see also
\cite{possanner18gfv} for a careful discussion of scalings and a rigorous analysis of this model.
It is assumed that ${|B_0(0)|\ge 1}$, that
 $B_0$, $B_1$ and $E$ are  smooth functions that  are bounded independently of $\eps$ on bounded domains together with all their derivatives, and that the initial data $x(0)=x^0$, $\dot x(0)=v^0$ are bounded independently of $\eps$.

In \eqref{ode}, $x(t)\in\real^3$ represents the position  at time $t$ of a charged particle (of unit mass and charge) that moves in the magnetic field $B$ and the electric field~$E$. The motion is composed of fast rotation around a guiding center (with the Larmor radius proportional to $\eps$) and slow motion of the guiding center.

The standard integrator for charged particles in a magnetic field is the {\it Boris algorithm} \cite{boris70rps} (see also, e.g., \cite{hairer18ebo}), which in the two-step formulation with step size $h$ is given by
\begin{equation}\label{boris}
\frac{x^{n+1} - 2x^n + x^{n-1}}{h^2} = \frac {x^{n+1} - x^{n-1}}{2h}  \times B(x^n,t^n) + E(x^n,t^n)
\end{equation}
with the velocity approximation
$v^n = \frac1{2h}\bigl( x^{n+1} - x^{n-1}\bigr) $ at time $t^n=nh$. This algorithm does, however, not behave well for \eqref{ode} with small $\eps$. Here we propose a modification, which we name {\it filtered Boris algorithm}. This modified integrator allows us to obtain better accuracy with considerably larger time steps, at minor additional computational cost. It is still a symmetric algorithm. We formulate and discuss this new algorithm in Section~\ref{sec:filtered-boris}. It comes in different variants that depend on the choice of a suitable filter function and of the positions where the magnetic field is evaluated, and we identify favourable choices.

In Section~\ref{sec:main-num} we state the main theoretical result, Theorem~\ref{thm:main}, which gives an error bound that behaves favourably with respect to $\eps$. While most filters only yield a first-order error bound in the positions, for a particular, non-trivial choice of the filter a second-order error bound is obtained. A second-order error bound is also obtained for the component of the velocity that is parallel to the magnetic field. For the normal velocity approximation, there is only a first-order error bound for any filter.  The proof of Theorem~\ref{thm:main} is based on comparing the modulated Fourier expansions of the exact and the numerical solutions, which are derived in Sections~\ref{sec:mfe-exact} and~\ref{sec:mfe-num}, respectively. Combining those results, the proof of Theorem~\ref{thm:main} is finally completed in Section~\ref{sec:proof}.

We remark that the differential equations for the coefficient functions of the modulated Fourier expansions derived  explicitly up to $\bigo(\eps^2)$ in Section~\ref{sec:mfe-exact} also yield the motion of the guiding center up to $\bigo(\eps^2)$. They coincide up to $\bigo(\eps^2)$ with the guiding center equations of the numerical approximation given by the filtered Boris integrator for an appropriate filter and for non-resonant step sizes $h\le C\eps$ with a possibly large constant~$C$. This does not hold true for the standard Boris integrator.

In Section~\ref{sec:boris-2} we describe a related, but different  integrator, called two-point filtered Boris algorithm, which evaluates the magnetic field both in the current position and in the current guiding center approximation in each step, and which has similar convergence properties to the previously considered filtered Boris method.

In Section~\ref{sec:num} we present the results of numerical experiments in which we compare the standard and filtered Boris algorithms.

In the Appendix we show how the filters are evaluated efficiently using a Rodriguez-type formula.

\section{Filtered Boris algorithm}
\label{sec:filtered-boris}

Using the velocity approximation at the mid-point,
\[
v^{n-1/2} = \frac 1h \bigl(x^{n} - x^{n-1} \bigr) = v^n  -
\frac{h}2\,v^n \times B(x^n,t^n) - \frac h2 E(x^n,t^n),
\]
the Boris algorithm \eqref{boris} is usually written and implemented as a one-step
method $(x^n,v^{n-1/2}) \mapsto (x^{n+1},v^{n+1/2})$,
\begin{equation}\label{boris-onestep}
\begin{array}{rcl}
v^{n-1/2}_+ &=&  v^{n-1/2} + \frac h2\, E(x^n,t^n) \\[2mm]
v^{n+1/2}_- - v^{n-1/2}_+  &=& \frac{h}2\, \bigl(v^{n+1/2}_- + v^{n-1/2}_+\bigr) \times B(x^n,t^n)\\[2mm]
v^{n+1/2} &=&  v^{n+1/2}_- + \frac h2\, E(x^n,t^n) \\[2mm]
x^{n+1} &=& x^n + h\,  v^{n+1/2} .
\end{array}
\end{equation}
To capture the high oscillations in the velocity more accurately, the second line
of \eqref{boris-onestep} needs to be modified, and one should rather work with
averaged velocities $ v^{n+1/2} \approx \tfrac1h\int_{t^n}^{t^{n+1}} v(t)\, \d t$ and possibly averaged positions.
This can be
achieved with the help of  filter functions like in \cite{garcia-archilla99lmf,hochbruck99agm} and
\cite[Section~XIII.2]{hairer06gni}.

For a vector $B=(b_1,b_2,b_3)^\top\in \real^3$ we denote by $|B|$ the Euclidean norm of $B$ and we use the common notation
\begin{equation}\label{skew-hat}
v\times B = - \widehat B \,v,\qquad \widehat B = \begin{pmatrix}
0 & - b_3 & b_2\\
b_3&0&-b_1\\
-b_2&b_1 & 0
\end{pmatrix} .
\end{equation}
For real-analytic functions $\Psi (\zeta )$ (such as $\exp (\zeta )$) we
 will form matrix functions
$\Psi(-h\widehat B)$, which are efficiently computed by a Rodriguez-type formula as described in the Appendix.
%

We denote by 
\begin{equation}\label{x-gc}
x^n_\odot = x^n + \frac{v^n \times B^n}{|B^n|^2}
\end{equation}
with $B^n=B(x^n,t^n)$ the guiding center approximation at time $t^n$ (cf.~\cite{northrop63tam}).
For the argument of~$B$ in the algorithm we choose a point on the straight line connecting $x^n$ and $x^n_\odot$: 
\begin{equation}\label{x-bar}
\bar x^n  = \theta^n x^n + (1-\theta^n) x^n_\odot
\end{equation}
with  $\theta^n=\theta(h |B^n|)$
for a  real function $\theta$. It turns out that there is a unique choice of $\theta$ such that a second-order error bound will be obtained:
\begin{equation}\label{theta}
\theta(\xi) = \frac1{\sinc(\xi/2)^2},
\end{equation}
where $\sinc(\xi)=\sin(\xi)/\xi$. 
We note that with the scaling \eqref{ode}, we have $\bar x^n = x^n +\bigo(\eps)$, provided that $h|B^n|$ is bounded away from non-zero integral multiples of $2\pi$.

We consider the following modification of the Boris algorithm.

\begin{algorithm}[Filtered Boris algorithm]
\label{alg:boris}
Given $(x^n , v^{n-1/2})$, the algorithm computes $(x^{n+1} , v^{n+1/2})$
as follows, with $B^n=B(x^n,t^n)$, $\bar B^n=B(\bar x^n,t^n)$ with $\bar x^n$ defined by \eqref{x-bar}, and $E^n=E(x^n,t^n)$:
\begin{equation}\label{filtered-boris-onestep}
\begin{array}{rcl}
v^{n-1/2}_+ &=&  v^{n-1/2} + \frac h2\, \Psi(h\widehat {B^n})\,E^n \\[2mm]
$$
v^{n+1/2}_- &=& \exp\bigl(-h\widehat {\bar B^n}\bigr) v^{n-1/2}_+.
$$
\\[2mm]
v^{n+1/2} &=&  v^{n+1/2}_- + \frac h2\, \Psi(h\widehat {B^n})\,E^n 
\\[2mm]
x^{n+1} &=& x^n + h\,  v^{n+1/2} ,
\end{array}
\end{equation}
where $\,\Psi(\zeta)=\mathrm{tanch}(\zeta/2)\,$ with $\,\mathrm{tanch}(\zeta)=\tanh(\zeta)/\zeta$.

The velocity approximation $v^n$ is computed as
\begin{equation}\label{boris-v}
v^n = \Phi_1(h\widehat {\bar B^n}) \,\frac{x^{n+1}-x^{n-1}}{2h} - h\Upsilon(h\widehat {B^n}) E^n ,
\end{equation}
where $~\displaystyle\Phi_1 (\zeta ) = \frac 1{\sinch (\zeta )}~$ with 
$~\displaystyle\sinch (\zeta ) = \frac{ \sinh (\zeta )}{\zeta}$, and
$~\displaystyle\Upsilon(\zeta ) =\frac{ \Phi_1 (\zeta ) -1 }{\zeta} $.
The starting approximation $v^{1/2}$ is computed from \eqref{boris-start} below with $n=0$.
\end{algorithm}

For the choice $\theta^n = 1$, the algorithm is explicit and requires only matrix-vector
multiplications that can be done efficiently with a Rodriguez-type formula (see the Appendix).

For $\theta^n = \theta (h|B^n|)$ with $\theta (\zeta ) $ from \eqref{theta},
the algorithm is implicit, because $\bar x^n$ then depends on $v^n$  and appears in the argument of $\bar B^n$ in the second line. This can be solved by a rapidly convergent fixed-point iteration for $\bar x^n$,
with  an error reduction by a factor $\bigo(\eps^2)$ in each iteration. 
We start the iteration
with $\bar x^n = x^n$, then compute $v_-^{n+1/2}$ from \eqref{filtered-boris-onestep}
and $v^n$ from \eqref{boris-v} using
\begin{equation}\label{v-pm}
\tfrac1{2h}(x^{n+1} - x^{n-1}) = \tfrac12\bigl(v^{n+1/2} + v^{n-1/2}\bigr) =
\tfrac12\bigl(v^{n+1/2}_- + v^{n-1/2}_+\bigr).
\end{equation}
This then yields $x^n_\odot$ from \eqref{x-gc} and the new $\bar x^n$ from \eqref{x-bar}.
In practice, one iteration is sufficient to get the improved accuracy. We note that all
matrix-vector multiplications can be done with a Rodriguez-type formula.

We mention that  Algorithm~\ref{alg:boris} preserves volume in phase space exactly
in the case of constant $B$ (and time-dependent $B(t)$),
but only approximately up to $\bigo(h\eps)$ in the general case of an inhomogeneous magnetic field \eqref{ode}.

\medskip
\noindent{\bf Two-step formulation.}
The filtered Boris algorithm has the symmetric two-step formulation
\begin{equation}\label{boris-twostep}
 \frac{x^{n+1} - 2x^n + x^{n-1}}{h^2}
  = 
\frac 2h\mathrm{tanh}\bigl( -\tfrac12 h\widehat {\bar B^n}\bigr) \,\frac{x^{n+1} - x^{n-1}}{2h}  +   \Psi( h  \widehat {B^n} ) E^n ,
\end{equation}
as is readily obtained by taking two consecutive steps and using \eqref{v-pm}.
This formulation is the basis of our theoretical analysis.

\medskip
\noindent{\bf Starting value.}
The starting value $v^{1/2}$ is chosen such that formulas
\eqref{filtered-boris-onestep}-\eqref{boris-v}
also hold for $n=0$. We find, for arbitrary $n$, that
\begin{equation}\label{boris-start}
 v^{n\pm 1/2} = \varphi_1 \bigl(\mp h \widehat {\bar B^n} \bigr) 
 \Bigl( v^n + h\Upsilon(h\widehat {B^n}) E^n\Bigr)
 \pm \frac h2  \Psi(h\widehat {B^n})\,E^n  ,
\end{equation}
where $\varphi_1(\zeta)=(e^\zeta-1)/\zeta$. Note that, for given $x^0$ and $v^0$, the
vectors $x^n_\odot$ and $\bar x^n$ are known, so that \eqref{boris-v} provides an explicit
 formula for $v^{1/2}$.

 
\medskip
\noindent{\bf One-step map ${(x^n,v^n)\mapsto (x^{n+1},v^{n+1})}$.}
Using the last formula of \eqref{filtered-boris-onestep} together with \eqref{boris-start} for relating $x^{n+1}$ and $x^n$,
and \eqref{boris-start} with $n$ and $+$ and with $n+1$ and $-$ for relating $v^{n+1}$ and $v^n$,
the filtered Boris algorithm can be written as
\begin{equation}
\label{expint-n}
\begin{array}{rcl}
x^{n+1} &=& \displaystyle
x^n + h \Phi^n_+ v^n
 +  \tfrac  {h^2}2 \,\Psi^n_+  E^n
 \\[3mm]
\Phi^{n+1}_- v^{n+1} &=& \Phi^{n}_+ v^{n} + \tfrac h2 \,\Psi^{n+1}_-  E^{n+1}  +  \tfrac h2 \,\Psi^n_+  E^n  ,
\end{array}
\end{equation}
where $\Phi^n_\pm = \varphi_1(\mp h\widehat {\bar B^n})$ and 
$\Psi^n_\pm =  \Psi (h\widehat {B^n} )\pm 2\Phi^n_\pm \Upsilon(h\widehat {B^n})$.

The method is symmetric in the sense that exchanging $n\leftrightarrow n+1$ and $h\leftrightarrow -h$
gives the same formulas.

\medskip
\noindent{\bf The integrator in the case of a constant magnetic field.}
For constant $B$,  we note that $(\Phi^{n+1}_-)^{-1} \Phi^{n}_+ = \exp(-h\widehat B)$, and so
\eqref{expint-n} reduces to the exponential integrator
(with the notation $\Psi_\pm(\zeta)=\Psi(\zeta) \mp 2\varphi_1(\pm\zeta)\Upsilon(\zeta)$)
\begin{equation}
\label{expint}
\begin{array}{rcl}
x^{n+1} &=& x^n + h \varphi_1 (-h\widehat B) v^n + \frac {h^2}2  \Psi_+ (-h\widehat B)
 E^n \\[1mm]
v^{n+1} &=& \exp (-h\widehat B) v^n + \frac h2 \bigl(\Psi_0 (-h\widehat B) E^n
+ \Psi_1 (-h\widehat B) E^{n+1} \bigr)
\end{array}
\end{equation}
with $\Psi_0(\zeta)=\Psi_+(\zeta)/\varphi_1(-\zeta)$ and $\Psi_1(\zeta)=\Psi_-(\zeta)/\varphi_1(-\zeta)$.
The method is exact for a constant magnetic field $B$ and vanishing
electric field $E$, because
\begin{equation} \label{exp-B}
\exp\left(
          \begin{array}{cc}
            0 & tI \\
            0 & -t \widehat B \\
          \end{array}
        \right) = \left(
                                                       \begin{array}{cc}
                                                             I & \ \ t\,\varphi_1(-t \widehat B) \\
                                                             0 & \ \exp(- t \widehat B) \\
                                                           \end{array}
                                                         \right).
\end{equation}
Since we have chosen $\Psi (\zeta ) =\mathrm{tanch}(\zeta/2)$,
the method is also {\it exact for constant $B$ and~$E$}. This is seen as follows:
For constant $B$, the
variation-of-constants formula for the system $\dot x=v, \ \ \dot v= x \times B + E(x)$ reads, in view of \eqref{exp-B},
$$
\begin{array}[c]{ll}x(t_n+h)=x(t_n)+ h\varphi_1(-h \widehat B)
v(t_n)
\\[1mm]
\qquad\qquad \qquad +\ h^2 \int_{0}^1(1- s )
\varphi_1(-(1- s )h \widehat B) E(x(t_n+h s ))  \d s ,\\[1mm]
v(t_n+h)=\exp(-h \widehat B)  v(t_n)+h \int_{0}^1
 \exp(-(1- s )h \widehat B) E(x(t_n+h s ))  \d s .
\end{array}
$$
For constant $E$, this becomes \eqref{expint},
which yields $\Psi_\pm(\zeta)=\varphi_2(\pm\zeta)$, where
$\varphi_2(\zeta)=(e^\zeta -1 - \zeta)/(\zeta^2/2) = \int_0^1 (1- s ) \,\varphi_1((1- s )\zeta)  \d s $.

\section{Statement of the main result}
\label{sec:main-num}

Our main theoretical result in this paper is the following error bound for the filtered Boris algorithm. Here we denote, for the exact velocity
$v(t)=\dot x(t)$,
$$
v_\parallel(t) = \frac{B(x(t),t)}{|B(x(t),t)|}\, \biggl(  \frac{B(x(t),t)}{|B(x(t),t)|}\cdot v(t) \biggr), \qquad v_\perp(t)=v(t)-v_\parallel(t),
$$
and similarly for the numerical velocity $v^n$,
$$
v_\parallel^n = \frac{B(x^n,t)}{|B(x^n,t^n)|}\, \biggl(  \frac{B(x^n,t^n)}{|B(x^n,t^n)|}\cdot v^n \biggr), \qquad v_\perp^n=v^n-v_\parallel^n.
$$
We then have the following result.

\begin{theorem}\label{thm:main} We assume the following, with arbitrarily chosen positive constants $c$, $C$, $M$ and $T$:
\begin{enumerate}
\item The initial velocity satisfies an $\eps$-independent bound
\begin{equation}\label{v-bound}
 |v^0| \le M.
\end{equation}
\item
The exact solution $x(t)$ of \eqref{ode} stays in a bounded set $K$ (independent of~$\eps$) for $0\le t \le T$.
\item
The step size satisfies $h \le C \eps$ and is such that the following non-resonance condition is satisfied:
\begin{equation}\label{non-res-exact}
\big|\sinc\bigl(\tfrac12 kh |B(x(t),t)|\bigr)\big| \ge c >0 \qquad\text{for } k=1,2,3.
\end{equation}
\end{enumerate}
If in the filtered Boris algorithm,
\begin{itemize}
\item $\bar x^n$ is given by \eqref{x-bar} with the function $\theta$ of \eqref{theta}, and
\item the filter functions $\Psi$ and $\Upsilon$ are defined as in Algorithm~\ref{alg:boris},
\end{itemize}
then the errors in the positions and the velocities are bounded by
\begin{equation}\label{err-2}
\begin{aligned}
x^n - x(t^n) &=\bigo(\eps^2), \\[0.5mm]
v_\parallel^n - v_\parallel(t^n) &=\bigo(\eps^2), \qquad v_\perp^n - v_\perp(t^n) =\bigo(\eps).
\end{aligned}
\end{equation}
For  a different choice of the functions $\theta$, $\Psi$ and $\Upsilon$, the error bounds are not better than $\bigo(\eps)$ for general problems~\eqref{ode}. 
The constants in the $\bigo$-notation are independent of $\eps$ and $h$
and $n$ with $0\le t^n=nh\le T$, but depend on~$T$, on the velocity
bound $M$ and the constants $c$ and $C$, and on bounds of
derivatives of $B_0$, $B_1$ and $E$ in a neighbourhood of the set
$K$.
\end{theorem}

We remark that in view of the error bounds, the non-resonance condition might be required along the numerical solution $x^n$ instead of the exact solution $x(t)$ as in \eqref{non-res-exact}.

The proof of this theorem will compare the modulated Fourier expansion of the exact solution (as given in Section~\ref{sec:mfe-exact}) with that of the numerical approximation (as given in Section~\ref{sec:mfe-num}). It will be given in Section~\ref{sec:proof}.

\begin{remark} The proof also shows that the choice $\bar x^n = x^n$ is sufficient for order~2 if the magnetic field satisfies, for all $z\in\complex^3$ and $x\in K$ and all times~$t$,
$$
\Im (z \times \partial_x B(x,t) \bar z) \cdot B(x,t) = \bigo(\eps).
$$
\end{remark}

\section{Modulated Fourier expansion of the exact solution}
\label{sec:mfe-exact}

We write the solution of \eqref{ode} as a {\it modulated Fourier expansion}
\begin{equation}\label{mfe-formal-exact}
x(t) \approx \sum_{k\in\integer} z^k (t) \,\e^{\iu k\phi (t)/\eps}
\end{equation}
with coefficient functions $z^k(t)$ for which all time derivatives are bounded independently of $\eps$,
where $\dot \phi (t) /\eps = \big| B\bigl(z^0(t),t\bigr)\big|$, and $z^0(t)$ describes the
motion of the guiding center.
Such a formal expansion has first been considered in \cite{kruskal58tgo}
for proving the existence of an adiabatic invariant (essentially the
magnetic moment $\tfrac12|\dot x \times B(x)|^2/|B(x)|^3$). It has been used for a rigorous proof of the long-time near-conservation of the magnetic moment in \cite{hairer19lta}, where this approach was extended to the numerical
solution of a variational integrator, for which near-conservation of the magnetic moment and of
the energy is rigorously proved over long times that cover arbitrary negative powers of $\eps$.

Following \cite{hairer19lta}, we diagonalize the linear
map $v\mapsto v\times B(x,t)$, which has eigenvalues
$\lambda_1 = \iu |B(x,t)|$, $\lambda_0 =0$, and $\lambda_{-1}=-\iu |B(x,t)|$.
We denote the normalized eigenvectors by $v_1(x,t), v_0(x,t),v_{-1}(x,t)$,
and remark that $v_0(x,t)$ is collinear to $B(x,t)$. We let
$P_j(x,t)=v_j(x,t)v_j(x,t)^*$ be the orthogonal projections onto the
eigenspaces. Furthermore, we write the coefficient functions of
\eqref{mfe-formal-exact} in the time-dependent basis $v_j\bigl(z^0(t),t\bigr)$,
\begin{equation}\label{zetak}
 z^k = z_1^k + z_0^k + z_{-1}^k ,\qquad z_j^k(t) =P_j\bigl(z^0(t),t\bigr)z^k(t) .
\end{equation}
Since $x(t)$ is real, we assume $z^{-k} = \overline{z^k}$ for all $k$.
Together with the fact that $v_{-1}(x,t) = \overline {v_1(x,t)}$ and
$v_0(x,t)$ is real, it follows
\begin{equation}\label{zetak-relation}
z_{-1}^{-k} = \overline{z_1^k},\qquad
z_0^{-k} = \overline{z_0^k},\qquad
z_{1}^{-k} = \overline{z_{-1}^k} .
\end{equation}
The following result is a variant of Theorem 4.1 in \cite{hairer19lta}, adapted to the present case of a strong magnetic field of the form  \eqref{ode}. Note that $B$ in this paper corresponds to $B/\eps$ in \cite{hairer19lta}.

\begin{theorem} \label{thm:mfe-exact}
Let $x(t)$ be a solution of \eqref{ode} with bounded initial velocity
\eqref{v-bound} that stays in a compact set $K$ for $0\le t \le T$.
For an arbitrary truncation index $N\ge 1$ we then have
\begin{equation}\label{mfe-remainder}
x(t) = \!\!\sum_{|k|\le N}\!\! z^k (t)\, \e^{\iu k \phi (t)/\eps}  + R_N (t) ,
\end{equation}
where the phase function satisfies $\dot \phi (t) = \eps | B (z^0(t),t ) |=\bigo(1)$.

(a) The coefficient functions $z^k(t)$ together with their derivatives (up to order
$N$) are bounded as $z_j^0 = \bigo (1)$ for $j\in \{-1,0,1\}$, $z_1^1 = \bigo (\eps )$,
$z_{-1}^{-1} = \bigo (\eps )$,
$z_j^k = \bigo (\eps^3)$ for $|k|=1$, $j\ne k$,
and for the remaining $(j,k)$ with $ |k| \le N$,
\begin{equation}\label{zk-est}
z_j^k = \bigo (\eps^{|k|+1}).
\end{equation}
They are unique up to $\bigo (\eps^{N+2})$. Moreover, we have
$\dot z^0 \times B(z^0,t) = \bigo (1 )$.

(b) The remainder term and its derivative are bounded by
\begin{equation}\label{remainder-est}
R_N(t) = \bigo (t^2 \eps^N), \quad \dot R_N(t) = \bigo (t \eps^N)
\quad\hbox{for}\quad 0\le t\le T.
\end{equation}

(c) The functions $z_0^0, z_{\pm 1}^0, z_1^1,z_{-1}^{-1}$ satisfy the
differential equations
\begin{align}
\ddot z_0^0 &=  P_0(z^0,t) E(z^0,t) + 2\,
P_0(z^0,t) \, \Re\! \Bigl( \iu \frac{\dot\phi}\eps z_1^1 \times B'(z^0,t)z_{-1}^{-1}\Bigr)\nonumber \\
& \qquad +  2\, \dot P_0(z^0,t) \dot z^0 + \ddot P_0(z^0,t) z^0 + \bigo (\eps^2 ) , \label{eq-z00}\\
\dot z_{\pm 1}^0 & =   \dot P_{\pm 1}(z^0,t)  z^0 \pm \iu \frac{\eps}{\dot\phi}
P_{\pm 1}(z^0,t) E(z^0,t) + \bigo (\eps^2), \label{eq-zpm0}\\[-1mm]
 \dot z_{\pm 1}^{\pm 1} &= - \frac{\ddot \phi}{\dot\phi}z_{\pm 1}^{\pm 1} + \bigo(\eps^2)=\bigo(\eps^2), \label{eq-zpmpm}
\end{align}
where we use the notation $\dot P_j(z^0,t) = \frac{\d}{\d t} P_j\bigl(z^0(t),t\bigr)$
and similar for $\ddot P_j(z^0,t)$.
All other coefficient functions $z_j^k$ are given by algebraic expressions depending
on $z^0, \dot z_0^0,z_1^1,z_{-1}^{-1}$.

(d) Assuming $\phi (0)=0$,
initial values for the differential equations of item~(c) are given by
\begin{align*}
z^0(0) &= x(0) + \frac{\dot x(0) \times B(x(0),0)}{|B(x(0),0)|^2} + \bigo(\eps^2),
\\
 \dot z_0^0(0) &=  P_0(x(0),0) \dot x(0)+ \dot P_0(x(0),0)  x(0)+ \bigo (\eps^2 ) , \\
  z_{\pm 1}^{\pm 1} (0) &= \mp \,\iu \frac{\eps}{\dot\phi (0)}P_{\pm 1}(x(0),0)\dot x(0) + \bigo(\eps^2).
\end{align*}
The constants symbolised by the \mbox{$\bigo$-notation} are independent
of $\eps$ and $t$ with $0\le t \le T$, but they depend on~$N$,
on the velocity bound $M$ in \eqref{v-bound}, on bounds of derivatives
of~$B$ and~$E$, and on~$T$.
\end{theorem}

\begin{remark}
We note that the guiding center motion of the system \eqref{ode} is given by the
non-oscillating term $z^0(t)$ in the modulated Fourier expansion. By the uniqueness of the modulated Fourier expansion up to high powers of~$\eps$, the equations in item (d) hold not only at time 0, but for all $t\le T$.
\end{remark}

\begin{proof}
(a) and (b): Compared to Theorem 4.1 in \cite{hairer19lta}, where a more general
strong magnetic field is considered,
the time interval of validity of the modulated Fourier expansion is here $\bigo(1)$ instead
of just $\bigo(\eps)$, and the bound \eqref{zk-est} is improved by a factor $\eps$.
The improvement of the time scale comes about by observing that a
function $ x_*(t)$ that solves \eqref{ode} up to a defect $d(t)$, i.e.,
$$
\ddot x_*(t) =  \dot x_*(t) \times B(x_*(t),t) + E(x_*(t),t) +d(t),
$$
satisfies an error bound, for $0\le t \le T$,
$$
|x_*(t) - x(t) | \le C \Bigl( |x_*(0)-x(0)| +  |\dot x_*(0)-\dot x(0)| + \int_0^t |d(t)| \, \d t\Bigr),
$$
where $C$ is independent of $\eps$ but grows exponentially with $T$.
This is proved by decomposing
$B(x,t)= \eps^{-1} B_0(0) + \eps^{-1}(B_0(\eps x)-B_0(0)) + B_1(x,t)$ and using
the variation-of-constants formula and the Gronwall inequality.
The improvement of the bound \eqref{zk-est} is a consequence of the fact that
the derivatives of $B(x,t)$ are bounded independently of $\eps$.

(c): For the error bound of Section~\ref{sec:proof} we need precise formulas for
the dominant terms of \eqref{mfe-remainder}.
Inserting the expansion \eqref{mfe-formal-exact} into the differential
equation \eqref{ode} and comparing the coefficients of $\e^{\iu k \phi (t)/\eps}$ yields
\begin{equation}\label{zetak-ode}
\ddot z^k + 2\iu k \frac{\dot\phi}\eps  \dot z^k+ \Bigl( \iu k \frac{\ddot\phi}\eps -
k^2 \frac{\dot\phi^2}{\eps^2}\Bigr) z^k = F^k,
\end{equation}
where, using Taylor series expansion for the nonlinearities,
\[
F^k =
\!\! \sum_{k_1+k_2 = k} \!\!\Bigl( \dot z^{k_1} + \iu k_1  \frac{\dot \phi}{\eps} z^{k_1}\Bigr)
 \times\!\!\! \sum_{m\ge 0 \atop s(\alpha ) = k_2} \!\!\!\!
 \frac {1}{m!} B^{(m)} (z^0,t) \,\bfz^\alpha
+ \!\!\sum_{m\ge 0 \atop s(\alpha ) = k}\!\!
\frac {1}{m!} E^{(m)} (z^0,t) \,\bfz^\alpha .
\]
Here, $B^{(m)}(x,t)$ and $E^{(m)}(x,t)$ denote the $m$th derivative with respect to $x$,
$\alpha =(\alpha_1,\ldots ,\alpha_m)$ is a multi-index with $\alpha_j \in \integer\setminus \{0\}$,
$s(\alpha ) = \alpha_1 + \ldots + \alpha_m$,
$|\alpha | = |\alpha_1| + \ldots + |\alpha_m|$, and $\bfz^\alpha = (z^{\alpha_1},
\ldots , z^{\alpha_m})$.

From \eqref{zetak-ode} it follows that the motion of the guiding center $z^0(t)$ is given by
\begin{equation}\label{zeta0-ode}
\ddot z^0 =  \dot z^0 \times B(z^0,t) + E(z^0,t) +
 2 \,\Re \Bigl( \iu \frac{\dot\phi}\eps z^1 \times B'(z^0,t)z^{-1}\Bigr)
+ \bigo (\eps^2 ) .
\end{equation}
The solution $z^0(t)$
is influenced by the functions $z^{\pm 1}$ which,  by \eqref{zetak-ode},
satisfy
\begin{equation}\label{zetapm1-ode}
\pm 2\iu  \frac{\dot\phi}\eps  \dot z^{\pm 1}+ \Bigl( \pm \iu \frac{\ddot\phi}\eps -
\frac{\dot\phi^2}{\eps^2}\Bigr) z^{\pm 1} =
\Bigl(\dot z^{\pm 1} \pm \iu \frac{\dot \phi}\eps
z^{\pm 1} \Bigr)  \times B(z^0,t)  + \bigo (\eps ) .
\end{equation}
Note that, whereas $B(z^0,t)$ is of size $\bigo (\eps^{-1})$,
its derivatives are bounded independently of $\eps$
due to the special form \eqref{ode}.

To get solutions with derivatives bounded uniformly in $\eps$, one has to
extract the dominant terms. Multiplying \eqref{zeta0-ode} with $P_0(z^0,t)$
eliminates the $\eps^{-1}$-term that is present in $B(z^0,t)$, and the second derivative
$\ddot z_0^0$ becomes dominant. Differentiating the relation $z_0^0 = P_0(z^0,t) z^0$
with respect to time yields
$\ddot z_0^0 = P_0(z^0,t) \ddot z^0 + 2\dot P_0(z^0,t) \dot z^0 + \ddot P_0(z^0,t) z^0$. This then gives \eqref{eq-z00}.
Note that, due to the special form of $B(x,t)$,
the time derivatives of $P_j(z^0,t)$ are of size $\bigo (\eps )$.

 A multiplication of
\eqref{zeta0-ode} with $P_{\pm 1}(z^0,t)$ gives
\[
P_{\pm 1}(z^0,t) \ddot z^0 = \pm \iu \frac {\dot\phi}\eps P_{\pm 1}(z^0,t) \dot z^0 +
P_{\pm 1}(z^0,t)E(z^0,t ) + \bigo (\eps ) .
\]
Substituting $P_{\pm 1}(z^0,t) \dot z^0 = \dot z_{\pm 1}^0 - \dot P_{\pm 1}(z^0,t)  z^0 $,
and extracting $\dot z_{\pm 1}^0$
yields \eqref{eq-zpm0}. Note that $\dot z_{\pm 1}^0 =\bigo (\eps )$, so that also
$\ddot z_{\pm 1}^0 =\bigo (\eps )$, and $ P_{\pm 1}(z^0,t) \ddot z^0 = \bigo (\eps )$.

Since $\dot\phi / \eps = | B(z^0,t)|$, the $\eps^{-2}$-terms
cancel in \eqref{zetapm1-ode} after projection with $P_{\pm 1}(z^0,t)$. Therefore, the
$\eps^{-1}$-terms are dominant and we obtain \eqref{eq-zpmpm}.

(d): Assuming $\phi (0)=0$, initial values are determined from \eqref{mfe-remainder} by
\begin{equation}\label{initial-val1}
\begin{array}{rcl}
x(0) &=& z^0 (0) +  \bigl( z^1 (0) + z^{-1} (0)\bigr) + \bigo (\eps^3) \\[1mm]
\dot x (0) &=&\displaystyle \dot z^0 (0) +  \bigl( \dot z^1 (0) + \dot z^{-1} (0)\bigr) +
\iu \frac{ \dot \phi (0) }\eps \bigl( z^1 (0) - z^{-1} (0)\bigr)  + \bigo (\eps^3 ) .
\end{array}
\end{equation}
This is a nonlinear system for $z^0(0), \dot z_0^0(0), z_1^1(0), z_{-1}^{-1}(0)$.
We write the vectors in the basis $\{v_j(z^0(0),0)\}$, and we select the dominant terms
in each equation. They are $z^0(0)$ in the upper relation of \eqref{initial-val1}, and
$\dot z_0^0 (0), z_1^1(0), z_{-1}^{-1}(0)$ in the lower relation. Fixed-point iteration then yields the
stated equations for the initial values. Note that the relation $P_j(z^0(0),0) =
P_j(x(0),0) + \bigo(\eps^2 )$ has been applied.
\qed
\end{proof}

\section{Modulated Fourier expansion of the numerical solution}
\label{sec:mfe-num}

We consider the two-step formulation \eqref{boris-twostep} of the filtered Boris algorithm,
and we write the numerical approximation $x^n$ as
\begin{equation}\label{mfe-formal-num}
x^n \approx \sum_{k\in\integer} z^k (t) \,\e^{\iu k\phi (t)/\eps} ,\qquad t=nh .
\end{equation}
We use the same notation for the coefficient functions as in Section~\ref{sec:mfe-exact}.
Note, however, that for the numerical solution these functions are not the same and depend
on the additional parameter~$h$. We again consider the basis $\{ v_j (x,t) \}$ and the corresponding
orthogonal projections $P_j(x,t)$, and we write the coefficient functions $z^k$ as in
\eqref{zetak}, with the only difference that here the argument $z^0(t)$ is the non-oscillating part
of \eqref{mfe-formal-num} and not that of \eqref{mfe-formal-exact}.

\begin{theorem}\label{thm:mfe-num}
Let $\{ x^n \}$ be a numerical solution of the filtered Boris algorithm
applied to \eqref{ode} with bounded
initial velocity \eqref{v-bound}, and suppose
that it stays in a compact set~$K$ for $0\le nh \le T$. We assume the non-resonance condition
\begin{equation}\label{non-res}
\big|\sinc\bigl(\tfrac12 kh |B(x^n,t^n)|\bigr)\big| \ge c >0 \qquad\text{for } k=1,\dots,N+1,
\end{equation}
for a fixed, but arbitrary truncation index $N\ge 2$, and (for convenience of presentation) also the bound $\eta = h/\eps \le C$. Moreover, we assume that the filter function $\Psi$ in \eqref{boris-twostep} is bounded by
$|\Psi(\iu\xi)| \le C  \,|\mathrm{tanc}(\tfrac 12 \xi)|$ for all real $\xi$,
where $\text{\rm tanc}(\xi)=\tan(\xi)/\xi$.
Then, we have
 that
\begin{equation}\label{mfe-remainder-num}
x^n = \!\!\sum_{|k|\le N}\!\! z^k (t)\, \e^{\iu k \phi (t)/\eps}  + R_N (t) ,\qquad t=nh,
\end{equation}
where the phase function is given by $\dot \phi (t) = \eps | B (z^0(t),t ) |$.

(a) and (b) The coefficient functions $z^k(t)$ together with their derivatives (up to order
$N$) as well as the remainder term and its derivative satisfy the bounds of items (a) and (b)
of Theorem~\ref{thm:mfe-exact}.

(c) The functions $z_0^0, z_{\pm 1}^0, z_1^1,z_{-1}^{-1}$ satisfy the
differential equations (with $\theta (\xi )$ used in the definition of $\bar x^n$ in \eqref{x-bar})
\begin{align}
\ddot z_0^0 &=  P_0(z^0,t) E(z^0,t) +
 2\,
P_0(z^0,t) \, \Re\! \Bigl( \iu \frac{\dot\phi}\eps z_1^1 \times B'(z^0,t)z_{-1}^{-1}\Bigr) \,\theta(\eta\dot\phi) \sinc(\eta\dot\phi/2)^2
 \nonumber
 \\[1mm]
& \qquad +  2\, \dot P_0(z^0,t) \dot z^0 + \ddot P_0(z^0,t) z^0 + \bigo (\eps^2 ) , \label{eq-z00-num}\\
\dot z_{\pm 1}^0 & =   \dot P_{\pm 1}(z^0,t)  z^0 \pm
\frac{\displaystyle \Psi \bigl( \iu {\eta \dot\phi} \bigr)}
{\displaystyle \text{\rm tanc} \Bigl(  \frac{\eta \dot\phi}2 \Bigr)} \,\iu \,\frac{\eps}{\dot\phi} \,
P_{\pm 1}(z^0,t) E(z^0,t) + \bigo (\eps^2), \label{eq-zpm0-num}\\[-1mm]
 \dot z_{\pm 1}^{\pm 1} &=  -\,\frac{1}{ \displaystyle\text{\rm tanc} 
 \Bigl(  \frac{\eta \dot\phi}2 \Bigr)}  \frac{\ddot\phi}{\dot\phi}\,z_{\pm 1}^{\pm 1}  + \bigo (\eps^2) = 
 \bigo(\eps^2) .\label{eq-zpmpm-num}
\end{align}

All other coefficient functions $z_j^k$ are given by algebraic expressions depending
on $z^0, \dot z_0^0,z_1^1,z_{-1}^{-1}$.

(d) Assuming $\phi (0)=0$,
initial values for the differential equations of item~(c) are given by the same equations as for the exact solution, up to $\bigo(\eps^2)$,
\begin{align}
\nonumber
z^0(0) &= x(0) + \frac{\dot x(0) \times B(x(0),0)}{|B(x(0),0)|^2} + \bigo(\eps^2),
\\
\label{initial-val-all}
 \dot z_0^0(0) &=  P_0(x(0),0) \dot x(0)+ \dot P_0(x(0),0)  x(0)+ \bigo (\eps^2 ) , \\
 \nonumber
  z_{\pm 1}^{\pm 1} (0) &= \mp \,\iu \frac{\eps}{\dot\phi (0)}P_{\pm 1}(x(0),0)\dot x(0) + \bigo(\eps^2).
\end{align}
The constants symbolised by the \mbox{$\bigo$-notation} are independent
of $\eps$ and $n$ with $0\le nh \le T$, but they depend on~$N$,
on the velocity bound $M$ in \eqref{v-bound}, on bounds of derivatives
of~$B$ and~$E$, and on~$T$.
\end{theorem}

\begin{proof} (a) and (b) We do not present the details of the proof of the existence of the modulated Fourier expansion and the bounds for the coefficient functions and the remainder term, since this uses the same kind of arguments as in previous such proofs, e.g. in \cite{hairer06gni,hairer16lta,hairer19lta}. In particular,
for $|k|=1, j\ne k$ and for $|k| \ge 2$ the construction of the coefficient functions
(see part (c) below) shows
that $z_j^k$ is multiplied by
\[
\frac 4 {\eta^2} \sin \Bigl( \frac{k\eta \dot\phi }2 \Bigr)
 \sin \Bigl( \frac{(k-j)\eta \dot\phi }2 \Bigr) .
\]
Under the non-resonance assumption \eqref{non-res} this expression is bounded from below
by a positive constant, so that an algebraic relation for $z_j^k$ can be extracted.

By construction of the coefficient functions the truncated series
of \eqref{mfe-remainder-num}
satisfies the two-step relation \eqref{boris-twostep} up to a defect of size $\bigo (\eps^N)$.
A standard discrete Gronwall argument then gives the bounds on the remainder.

(d) The initial values are obtained from
\begin{equation}\label{initial1-num}
x(0) = z^0(0) + \bigl( z^1(0) + z^{-1}(0)\bigr) + \bigo (\eps^2),
\end{equation}
which is a consequence of \eqref{mfe-formal-num}, and from
\begin{align}\nonumber
\dot x(0) &= \Phi_1\bigl(h\widehat B(x(0),0)\bigr) \dot z^0(0) + \frac{\iu\dot\phi(0)}\eps\, z^1_1(0) - \frac{\iu\dot\phi(0)}\eps\, z^{-1}_{-1}(0)
\\
&\quad -
h\Upsilon\bigl(h\widehat B(x(0),0)\bigr) E(x(0),0) + \bigo(\eps^2),
\label{initial2-num}
\end{align}
which follows from \eqref{boris-v} and Lemma~\ref{lem:differences}.
As in the proof of Theorem~\ref{thm:mfe-exact} this constitutes a nonlinear system for the
values $z^0(0), \dot z_0^0(0), z_1^1 (0), z_{-1}^{-1}(0)$. The relation \eqref{initial1-num}
yields $z^0_0(0)$. Multiplication of \eqref{initial2-num} with $P_j\bigl( z^0(0),0\bigr) =
P_j \bigr( x(0),0 \bigr) + \bigo (\eps^2 )$ gives $\dot z_0^0(0)$ for $j=0$ and
$z_{\pm 1}^{\pm 1}(0) $ for $j=\pm 1$, where we use in addition that $\Phi_1(h\widehat B(x(0),0))=\Phi_1(h\widehat B(z^0(0),0))+\bigo(\eps^2)$ and
$P_{\pm 1}\dot z^0=\dot z^0_{\pm1} - \dot P_{\pm1} z^0 = \bigo(\eps)$.
Remarkably we get, up to terms of size $\bigo (\eps^2 )$, the same formulas for the
initial values as for the exact solution.

By the uniqueness of the modulated Fourier expansion (up to $\bigo(\eps^N)$),
these relations  hold not only at time $0$, but for arbitrary times $t\le T$, except for a phase factor $e^{\mp\iu\phi(t)}$ in the equation for $z^{\pm1}_{\pm1}$. This phase factor did not appear in \eqref{initial-val-all} because we chose $\phi(0)=0$.

(c) To derive the differential equations for the coefficient functions we first expand the
perturbed argument of $B(x,t)$ in the filtered Boris algorithm as
\begin{equation}\label{mfe-formal-xbar}
\bar x^n \approx \sum_{k\in\integer} \zeta^k (t) \,\e^{\iu k\phi (t)/\eps} ,\qquad t=nh .
\end{equation}
The coefficient functions $\zeta^k(t)$ are obtained as follows: inserting the
modulated Fourier expansion \eqref{mfe-formal-num} into \eqref{boris-v},
using Lemma~\ref{lem:differences} below, and replacing $\Phi_1\bigl( h\widehat{\bar B^n}\bigr)$
by $\Phi_1\bigl( h\widehat B (z^0(t^n),t^n)\bigr)$ yields with $t=nh$
\[
v^n = \dot z_0^0 (t) + \frac {\iu \dot \phi (t)}{\eps} z_1^1(t)\,\e^{\iu \phi (t)/\eps}
- \frac {\iu \dot \phi (t)}{\eps} z_{-1}^{-1}(t)\,\e^{-\iu \phi (t)/\eps} + \bigo (\eps ) ,
\]
see also the more detailed computation in Section~\ref{sec:proof}.
Since we have $\dot z_0^0 (t) = P_0\bigl( z^0(t),t\bigr) z^0(t) + \bigo (\eps )$
and $B^n = \widehat B \bigl(z^0(t^n),t^n\bigr)+ \bigo (\eps)$, this implies
\[
\frac{v^n\times B^n}{|B^n|^2} = - z_1^1(t)\,\e^{\iu \phi (t)/\eps}
-  z_{-1}^{-1}(t)\,\e^{-\iu \phi (t)/\eps} + \bigo (\eps^2 ) ,
\]
and consequently $x^n_\odot = z^0 (t^n) + \bigo (\eps^2 )$, which shows that 
$x^n_\odot $ is an excellent approximation of the non-oscillating part of the
numerical solution $x^n$. Together with the definition \eqref{x-bar} of $\bar x^n$
we find the dominating terms of the expansion \eqref{mfe-formal-xbar} as
\begin{equation}\label{zeta0-zeta1}
\zeta^0(t) =z^0(t) + \bigo (\eps^2 ), \qquad \zeta^{\pm 1} (t) = \theta (t)\, z^{\pm 1} (t) + \bigo (\eps^2 ),
\end{equation}
where $\theta(t)=\theta(h\dot\phi(t)/\eps)=\theta\bigl(h|B(z^0(t),t)|\bigr)$.

After this preparation,
we insert the expansions \eqref{mfe-formal-num} for $x^n$
and \eqref{mfe-formal-xbar} for $\bar x^n$ into the two-step formulation
\eqref{boris-twostep} of the filtered Boris algorithm.
Using Lemma~\ref{lem:differences} below,
expanding the nonlinearities around~$\zeta^0$ and $z^0$,
and comparing the coefficients of $\e^{\iu k \phi (t)/\eps}$ yields
\[
\begin{array}{rcl}
\displaystyle \sum_{l\ge 0} \eps^{l-2} d_l^k \frac{\d^l}{\d t^l} z^k &=& \displaystyle
\sum_{k_1+k_2 = k} \! \Bigl( \!\sum_{m\ge 0 \atop s(\alpha ) = k_1} \!\!\!
 \frac {1}{m!} T_{\widehat B}^{(m)} (\zeta^0,t) \,\bfzeta^\alpha \Bigr)
\Bigl( \,\sum_{l\ge 0} \eps^{l-1} c_l^{k_2} \frac{\d^l}{\d t^l} z^{k_2} \Bigr) \cr
&+& \displaystyle
\sum_{k_1+k_2 = k} \!\Bigl(\! \sum_{m\ge 0 \atop s(\alpha ) = k_1} \!\!\!
 \frac {1}{m!} \Psi_{\widehat B}^{(m)} (z^0,t) \,\bfz^\alpha \Bigr)
\Bigl( \! \sum_{m\ge 0 \atop s(\alpha ) = k_2} \!\!\!
 \frac {1}{m!} E^{(m)} (z^0,t) \,\bfz^\alpha  \Bigr) ,
\end{array}
\]
where $T_{\widehat B}^{(m)} (x,t)$ denotes the $m$th derivative
of $T_{\widehat B} (x,t) = \frac 2h \tanh \bigl( -\frac h2 \widehat B (x,t)\bigr)$
with respect to $x$ and, similarly, $\Psi_{\widehat B}^{(m)} (x,t)$ is the $m$th derivative
of $\Psi_{\widehat B} (x,t) = \Psi \bigl( -h \widehat B (x,t)\bigr)$ with respect to $x$.
These derivatives are bounded under the assumption that
$\eta = h/\eps \le c$ and the non-resonance condition~\eqref{non-res-exact}.

For $k=0$ we obtain
\begin{equation}\label{zpp-num}
\begin{array}{rcl}
\ddot z^0 &=& \displaystyle T_{\widehat B}(\zeta^0,t) \dot z^0 + \Psi_{\widehat B} (z^0,t) E (z^0,t) \\[1mm]
 &+& \displaystyle
 2\, \Re \Bigl( \bigl(T_{\widehat B}' (\zeta^0,t) \zeta^{-1} \bigr)\frac {\iu}{\eps \eta } 
 \sin (\eta \dot \phi ) z^1 \Bigr)
+ \bigo (\eps^2 ) ,
\end{array}
\end{equation}
and for $k=\pm1 $ we get
\begin{equation}\label{z1p-num}
\eps^{-2} d_0^{\pm 1} z^{\pm 1} + \eps^{-1} d_1^{\pm 1} \dot z^{\pm 1} =
T_{\widehat B}(\zeta^0 ,t) \bigl( \eps^{-1} c_0^{\pm 1} z^{\pm 1} + c_1^{\pm 1} \dot z^{\pm 1} 
\bigr) + \bigo (\eps ) .
\end{equation}
Because of \eqref{zeta0-zeta1}, the argument $\zeta^0$ can be replaced by $z^0$
in these equations.
In the limit $h\to 0$, i.e., $\eta\to 0$ we have accordance with the equations
\eqref{zeta0-ode} and
\eqref{zetapm1-ode} for the exact solution, respectively.

To get the differential equations for the dominant coefficient functions, we shall use
the relations
\[
\begin{array}{l}
\displaystyle P_0(z^0,t) T_{\widehat B}(z^0,t) = 0, \qquad
P_{\pm 1}(z^0,t) T_{\widehat B}(z^0,t) = \pm\, \iu \frac 2h \tan \Bigl( \frac {h \dot \phi }{2\eps} \Bigr)
P_{\pm 1}(z^0,t)  , \\[2mm]
\displaystyle P_0(z^0,t) \Psi_{\widehat B}(z^0,t) = P_0(z^0,t), \qquad
P_{\pm 1}(z^0,t) \Psi_{\widehat B}(z^0,t) =  \Psi \bigl(\pm \iu \frac {h \dot \phi }{\eps} \Bigr)
P_{\pm 1}(z^0,t)   .
\end{array}
\]
Multiplying the equation \eqref{zpp-num} with $P_0(z^0,t)$ and applying the differentiation
formula of Lemma~\ref{lem:tanh-deriv} yields the differential equation
\begin{align}
P_0(z^0,t) \,\ddot z_0^0 &=  P_0(z^0,t) E(z^0,t)  \nonumber
 \\[0mm]
&\hspace{-10mm}+ 2\,
P_0(z^0,t) \, \Re\!\biggl(  -\frac2{\eta\dot\phi}\tan \Bigl(\frac{\eta\dot \phi}2\Bigr) \Bigl( \widehat B'(z^0,t)\zeta_{-1}^{-1}\Bigr)\, \iu \frac{\dot\phi}\eps
\sinc (\eta\dot\phi) z_1^1 \biggr)) 
+ \bigo (\eps^2 ) .  \nonumber
\end{align}
Using $\bigl( \widehat B ' (x,t) \Delta x\bigr) v = - v\times B ' (x,t) \Delta x$, which follows from
differentiation of $\widehat B(x,t) v = - v\times B(x,t)$, the trigonometric identity
$\sin (2\alpha ) = 2 \sin (\alpha) \cos (\alpha )$, and the second relation of \eqref{zeta0-zeta1},
this equation becomes \eqref{eq-z00-num}.

A multiplication of \eqref{zpp-num} with $P_{\pm 1} (z^0,t)$ permits to extract the dominant
first derivative $\dot z_{\pm 1}^0$ and  gives \eqref{eq-zpm0-num}.

We next consider the equation \eqref{z1p-num}. The $\eps^{-2}$-terms in the left and right
sides are contained in
\[
- \frac 4{\eps^2 \eta^2} \sin^2 \Bigl( \frac{\eta\dot\phi}2\Bigr)z^{\pm 1}\qquad\hbox{and}\qquad
\pm \frac 2{\eps h} \tanh \Bigl( - \frac h2 \widehat B(z^0,t)\Bigr) \frac \iu \eta \sin (\eta \dot \phi ) z^{\pm 1} .
\]
After multiplication with $P_{\pm 1 }(z^0,t)$ these terms cancel because of the
above formula for $P_{\pm 1}(z^0,t) T_{\widehat B}(z^0,t)$. The remaining terms lead to
\[
\dot z_{\pm1}^{\pm1} =  -\frac
{\displaystyle \Bigl( \cos (\eta\dot \phi ) + 
\tan \Bigl( \frac{\eta \dot \phi } 2 \Bigr) \sin (\eta \dot\phi ) \Bigr) \ddot \phi}
{\displaystyle  \frac 2\eta \Bigl( \sin (\eta\dot \phi ) -
\tan \Bigl( \frac{\eta \dot \phi } 2 \Bigr) \cos (\eta \dot\phi )\Bigr)}
z_{\pm1}^{\pm1} + \bigo (\eps^2 ) ,
\]
which simplifies to \eqref{eq-zpmpm-num}.
\qed
\end{proof}

In the above proof we referred to the following lemmas.

\begin{lemma}[\cite{hairer19lta}]\label{lem:differences}
For smooth functions $\phi (t)$ and $z^k(t)$ let $y^k(t) = \e^{\iu k \phi (t)/\eps}  z^k(t)$,
and denote $\eta = h/\eps$.
The finite differences of $y^k(t)$ then satisfy
\[
\begin{array}{rcl}
\delta_{2h}y^k(t) &=& \displaystyle
\frac{y^k(t+h) - y^k(t-h)}{2h}  ~=~ \e^{\iu k \phi (t)/\eps} \sum_{l\ge 0}
\eps^{l-1} c_l^k(t)  \frac{\d^l}{\d t^l} z^k(t)\\[4mm]
\delta_h^2 y^k(t) &=& \displaystyle
\frac{y^k(t+h) - 2y^k(t) + y^k(t-h)}{h^2} ~=~ \e^{\iu k \phi (t)/\eps} \sum_{l\ge 0}
\eps^{l-2} d_l^k(t)  \frac{\d^l}{\d t^l} z^k(t) ,
\end{array}
\]
where $c_{2j}^0 =0$, $c_{2j+1}^0 = \eta^{2j} /(2j+1)!$, and $d_0^0=0$,
$d_{2j}^0 = 2\eta^{2j-2}/(2j)!$, $d_{2j+1}^0 = 0$. The leading coefficients are
\begin{equation}\label{coeff-c-d}
\begin{array}{rcl}
c_0^k(t) &=& \displaystyle \frac{\iu}{\eta} \sin \bigl( k\eta \dot \phi (t)\bigr) -
\eps \frac{k\eta}2  \sin \bigl( k\eta \dot \phi (t)\bigr) \ddot \phi (t)+ \bigo (\eps^2 ) \\[4mm]
c_1^k(t) &=& \displaystyle \cos \bigl( k\eta \dot \phi (t)\bigr) + \bigo (\eps ) \\[2mm]
d_0^k(t) &=& \displaystyle -\frac{4}{\eta^2} \sin^2 \Bigl( \frac{k\eta \dot \phi (t)}2\Bigr) +
\iu\, \eps\, k \cos \bigl( k\eta \dot \phi (t)\bigr)  \ddot \phi (t) + \bigo (\eps^2 ) \\[4mm]
d_1^k(t) &=&\displaystyle  \frac {2\,\iu}\eta \sin \bigl( k\eta \dot \phi (t)\bigr) + \bigo (\eps ) .
\end{array}
\end{equation}
\end{lemma}
Note that these coefficients depend on $\eta$, $\eps$, and $t$ via derivatives of $\phi (t)$.

\begin{proof}
Expanding $\phi (t\pm h)$ and $z^k(t\pm h)$ into Taylor series around $t$ yields the
stated formulas.
\qed
\end{proof}

\begin{lemma}\label{lem:tanh-deriv}
Let $T_{\widehat B} (x,t) = \frac 2h \tanh \bigl( -\frac h2 \widehat B (x,t)\bigr)$,
and let $P_j(x,t)$ be the orthogonal projections onto the eigenspace
of $\widehat B (x,t)$ corresponding to the eigenvalues $\lambda_0 =0$ and
$\lambda_1 = - \iu |B(x,t)| = - \iu \dot \phi (x,t)/\eps$, and
$\lambda_{-1} =  \iu |B(x,t)| =  \iu \dot \phi (x,t)/\eps$,
respectively. Omitting the argument $(x,t)$, we then have with $\eta = h/\eps$,
\[
P_0 \Bigl( T_{\widehat B}' \Delta x \Bigr) P_{\pm 1} =
\mp \frac 2{\eta \dot\phi } \tan \Bigl( \frac{\eta \dot \phi}2 \Bigr) 
P_0 \Bigl( \widehat B' \Delta x \Bigr) P_{\pm 1} ,
\]
where prime indicates the derivative with respect to $x$.
\end{lemma}

\begin{proof}
Writing $\tanh$ as a Taylor series with coefficients $\gamma_l$ and differentiating
term by term, we obtain
\[
\begin{array}{rcl}
P_0 \Bigl( T_{\widehat B} ' \Delta x \Bigr) P_{\pm 1} &=&\displaystyle \frac 2h \sum_{l\ge 1} \gamma_l
\Bigl( - \frac h2 \Bigr)^l P_0  \Bigl( \widehat B ' \Delta x \Bigr) \widehat B^{l-1} P_{\pm 1} \\
&=& \displaystyle\frac 2h \sum_{l\ge 1} \gamma_l
\Bigl( - \frac h2 \Bigr)^l P_0  \Bigl( \widehat B ' \Delta x \Bigr)
 \Bigl(\mp \iu\, \frac {\dot\phi}\eps \Bigr)^{l-1}P_{\pm 1}\\
&=& \displaystyle \frac{2\iu}{\eta \dot \phi} \tanh \Bigl( \pm\iu \frac{\eta \dot\phi }{2} \Bigr) P_0 
\Bigl( \widehat B ' \Delta x \Bigr) P_{\pm 1} .
\end{array}
\]
This proves the statement of the lemma.
\qed
\end{proof}

\section{Proof of Theorem~\ref{thm:main}}
\label{sec:proof}


Theorems~\ref{thm:mfe-exact} and~\ref{thm:mfe-num} show that the coefficient functions $z^k(t)$ (and also $\dot z^0(t)$) of the modulated Fourier expansions of the exact and numerical solutions coincide up to $\bigo(\eps^2)$ for the choice \eqref{theta} and $\,\Psi(\zeta)=\mathrm{tanch}(\zeta/2)$. This also shows that the phase functions $\phi$ (with $\dot \phi(t)= \eps |B(z^0(t),t)|$) differ only by $\bigo(\eps^2)$, respectively. Since all coefficient functions $z^k$ of the modulated Fourier expansion with the exception of $z^0$ are of size $\bigo(\eps)$ or smaller, this yields that all summands $z^k(t)\e^{\iu k\phi(t)/\eps}$ still differ only by $\bigo(\eps^2)$.
So we obtain the $\bigo(\eps^2)$ error bound for the positions as stated in Theorem~\ref{thm:main}.

We now turn to the error bound for the velocities.
By Theorem~\ref{thm:mfe-exact}, using that $\dot z^{\pm 1}_{\pm 1}=\bigo(\eps^2)$ and $z^k_j=\bigo(\eps^3)$ for $|k|=1$ and $k\ne j$ and for $|k|\ge 2$ and all $j=-1,0,1$, together with their derivatives, the velocity of the exact solution satisfies
\begin{equation}\label{v-mfe}
v(t)=\dot x(t) = \dot z^0(t) + \frac{\iu\dot\phi(t)}\eps\, z^1_1(t) \, \e^{\iu\phi(t)/\eps} - \frac{\iu\dot\phi(t)}\eps\, z^{-1}_{-1}(t) \, \e^{-\iu\phi(t)/\eps} + \bigo(\eps^2).
\end{equation}
We shall show below that the numerical solution admits  the same expansion with functions
$\phi (t), \dot z^0(t)$, $z^1_1(t)$, $z^{-1}_{-1}(t)$ that correspond to the modulated Fourier expansion
\eqref{mfe-formal-num} of the numerical solution and not to \eqref{mfe-formal-exact} of the exact solution.
By Theorems~\ref{thm:mfe-exact} and \ref{thm:mfe-num}, these functions differ only
by $\bigo (\eps^2)$.
Because of the denominator $\eps$ in the second and third terms on the
right-hand side of \eqref{v-mfe}, this yields
\begin{equation} \label{v-err-mfe}
v^n - v(t^n)= \bigo(\eps),\qquad \text{but}\qquad v^n_\parallel - v_\parallel(t^n)= \bigo(\eps^2),
\end{equation}
and proves the statement of Theorem~\ref{thm:main}.

Using Lemma~\ref{lem:differences} and $\ddot\phi(t)=\bigo(\eps)$, we have, with $t=nh$, that
\[
\frac{x^{n+1} - x^{n-1}}{2h} = \dot z^0 (t) + \sinc\bigl( \eta \dot\phi (t)\bigr)
\frac{\iu \dot \phi (t)}{\eps} \Bigl( z_1^1(t) \e^{\iu \phi (t)/\eps} - z_{-1}^{-1}(t) \e^{-\iu \phi (t)/\eps} \Bigr)
+ \bigo (\eps^2 ) .
\]
A consequence of the maximal ordering in \eqref{ode} is that $\Phi_1\bigl( h \widehat B(\bar x^n,t^n)\bigr)
=  \Phi_1\bigl( h \widehat B(z^0(t^n),t^n)\bigr) + \bigo (\eps^2)$, and
$\Upsilon \bigl( h \widehat B( x^n,t^n)\bigr)
= \Upsilon\bigl( h \widehat B(z^0(t^n),t^n)\bigr) + \bigo (\eps^2)$.
Splitting $\Phi_1(\cdot ) \dot z^0$ into $\dot z^0 + \bigl(\Phi_1(\cdot ) - I \bigr)\dot z^0$
and using $\Upsilon (\zeta ) = \bigl( \Phi_1 (\zeta ) -1\bigr)/\zeta$, we
therefore have
\begin{align*}
v^n &= \Phi_1 \bigl( h\widehat B(\bar x^n,t^n)\bigr) \frac{x^{n+1} -x^{n-1}}{2h} -
h\Upsilon \bigl(h\widehat B( x^n,t^n)\bigr) E(x^n,t^n)\\
&= \dot z^0(t) + 
\frac{\iu \dot \phi (t)}{\eps} \Bigl( z_1^1(t) \e^{\iu \phi (t)/\eps} - z_{-1}^{-1}(t) \e^{-\iu \phi (t)/\eps} \Bigr)\\
&\hspace{12.5mm} + h\Upsilon\bigl( h \widehat B(z^0(t),t)\bigr) \Bigl(\widehat B\bigl(z^0(t),t\bigr) \dot z^0(t)
-E\bigl(z^0(t),t\bigr)\Bigr)  + \bigo (\eps^2 ) .
\end{align*}
Since $\Upsilon (0)=0$ we have $\Upsilon\bigl( h \widehat B(z^0(t),t)\bigr)P_0(z^0(t),t) = 0$.
On the other hand
\[
P_{\pm 1}(z^0(t),t) \Bigl(\widehat B\bigl(z^0(t),t\bigr) \dot z^0(t)
-E\bigl(z^0(t),t\bigr)\Bigr) = \bigo (\eps^2 )
\]
which follows from \eqref{eq-zpm0-num} for $\Psi (\iu y) = \text{\rm tanc}  (y/2)$. This proves
the relation \eqref{v-mfe} also for the numerical solution.

\section{A two-point filtered Boris algorithm}
\label{sec:boris-2}
Algorithm~\ref{alg:boris} evaluates the magnetic field $B$ at $\bar x^n$ given by \eqref{x-bar}--\eqref{theta}, which can be far from both $x^n$ and the guiding center approximation $x^n_\odot$ of \eqref{x-gc} when $h|B(x_n)|$ is close to a nonzero integral multiple of $2\pi$. In the following we propose an alternative filtered Boris algorithm with the same second-order convergence properties as in Theorem~\ref{thm:main}, which evaluates the magnetic field at the two points $x^n$ and $x^n_\odot$.

\begin{algorithm}[Two-point filtered Boris algorithm]
\label{alg:boris-2}
Given $(x^n , v^{n-1/2})$, the algorithm computes $(x^{n+1} , v^{n+1/2})$
as follows, with $B^n=B(x^n,t^n)$, $B^n_\odot=B(x^n_\odot,t^n)$ and $E^n=E(x^n,t^n)$:
\begin{equation}\label{filtered-boris-onestep2}
\begin{array}{rcl}
v^{n-1/2}_+ &=&  v^{n-1/2} + \frac h2\, \Psi(h\widehat B^n)\,E^n \\[2mm]
\Phi_2(h\widehat B^n_\odot) (v^{n+1/2}_- - v^{n-1/2}_+)  &=& 
\frac{h}2\,  \Phi_1(h\widehat B^n)\bigl(v^{n+1/2}_- + v^{n-1/2}_+\bigr) \times B^n\\[3mm]
v^{n+1/2} &=&  v^{n+1/2}_- + \frac h2\, \Psi(h\widehat B^n)\,E^n \\[2mm]
x^{n+1} &=& x^n + h\,  v^{n+1/2} ,
\end{array}
\end{equation}
where $\Psi (\zeta ) = \tanch (\zeta/2)$ and
$\displaystyle\Phi_1(\zeta ) = \frac 1{\sinch (\zeta )}$ are as in
Algorithm~\ref{alg:boris}, and $\displaystyle\Phi_2 (\zeta ) = \frac 1{\sinch(\zeta /2)^2} $.

The velocity approximation $v^n$ is again computed by \eqref{boris-v}, with $B^n$ instead of $\bar B^n$.
\end{algorithm}

For constant $B$, Algorithms~\ref{alg:boris} and~\ref{alg:boris-2} are identical and explicit.
In the general case, both methods are implicit, but this time the fixed-point iteration for $x^n_\odot$  requires not only  the evaluation of matrix functions by the Rodriguez formula,
but in addition the solution of a linear system with the $3\times 3$ matrix $ \Phi_2(h\widehat B^n_\odot)+ \frac 12  h \widehat B^n \Phi_1(h\widehat B^n)$.
We further note that in the case of a vanishing electric field, $E^n=0$, Algorithm~\ref{alg:boris} preserves the velocity norm 
$|v^{n+1/2}|=|v^{n-1/2}|$, which is satisfied only approximately up to $\bigo(h\eps)$ by Algorithm~\ref{alg:boris-2}. While these properties are unfavourable for Algorithm~\ref{alg:boris-2}, our numerical experiments indicate that it yields higher accuracy than Algorithm~\ref{alg:boris-2} for stepsizes such that $h|B|$ is large, and in particular it is less sensitive to near-resonances where $h|B|$ is close to an integral multiple of $2\pi$. 

The {\it two-step formulation} of Algorithm~\ref{alg:boris-2} is
\begin{align}
 &\frac{x^{n+1} - 2x^n + x^{n-1}}{h^2}
 \label{boris-2-twostep}
 \\
 \nonumber
&\quad = \,\Phi_2( h\widehat B(\bar x^n,t^n))^{-1}\Bigl(\Phi_1( h\widehat B^n) \frac{x^{n+1} - x^{n-1}}{2h}\times B^n\Bigr)+   \Psi( h  \widehat B^n ) E^n .
\end{align}

The {\it starting value} $v^{1/2}$ is chosen such that formulas
\eqref{filtered-boris-onestep2} and \eqref{boris-v}
also hold for $n=0$. With the abbreviations
\begin{align*}
& \Lambda^n =  \Phi_2( h\widehat B^n_\odot)^{-1}\Phi_1( h\widehat B^n), \ \Psi^n = \Psi (h\widehat B^n ),\ \Upsilon^n=\Upsilon(h\widehat B^n),
\\
&\Phi^n_\pm =
( I \mp
 \tfrac 1 2 \,  \Lambda^n h\widehat B^n )  \sinch (h\widehat B^n) ,
 \\[1mm]
 &\Psi^n_\pm = \Psi^n \pm 2\Phi^n_\pm \Upsilon^n,
\end{align*}
we find, for arbitrary $n$, that like in \eqref{boris-start},
$$
 v^{n\pm 1/2} = \Phi^n_\pm v^n
 \pm \tfrac h 2 \,\Psi^n_\pm  E^n,
$$
and the one-step map ${(x^n,v^n)\mapsto (x^{n+1},v^{n+1})}$ is then again given by \eqref{expint-n} with these modified matrices $\Phi^n_\pm$ and $\Psi^n_\pm$.

For the two-point filtered Boris algorithm, the second-order convergence result of Theorem~\ref{thm:main} in $x$ and $v_\parallel$ and the first-order convergence in $v_\perp$ remain valid, as can be shown by an adaptation of the proof of Theorem~\ref{thm:mfe-num}, for which we omit the details.

\section{Numerical experiment}
\label{sec:num}

As an illustrative
numerical experiment, we consider  the charged-particle motion in
the magnetic field
$$B(x,t) =\nabla \times \frac 1 \eps \,  \left(
                              \begin{array}{c}
                             0 \\
                             x_1\\
                              0 \\
                              \end{array}
                            \right)+\nabla \times  \,  \left(
                              \begin{array}{c}
                            0  \\
                            x_1 x_3 \\
                              0  \\
                              \end{array}
                            \right)
                            =\frac 1 \eps \,  \left(
                              \begin{array}{c}
                              0\\
                              0 \\
                              1 \\
                              \end{array}
                            \right)
 +  \left(
     \begin{array}{c}
      -  x_1 \\
       0  \\
      x_3 \\
     \end{array}
   \right),
 $$
 and the electric field $E(x,t)=-\nabla_x U(x)$ with the potential
$$U(x)=\frac{1}{\sqrt{x_1^2+x_2^2}}.$$
 The initial values are chosen as $x(0)=(\frac 1 3,\frac 1 4,\frac 1
2)^{\intercal}$ and $v(0)=(\frac 2 5,\frac 2 3,1)^{\intercal}$. We
solve this problem for $0\le t \le 1$ with
$h=\epsilon,4\epsilon,16\epsilon$ and compare the numerical errors of the following methods:
\begin{itemize}
\item the standard Boris algorithm, 
\item Exp-A: the filtered Boris method of Algorithm~\ref{alg:boris} with $\theta=1$ in \eqref{x-bar} (where $\bar x^n=x^n$ and the method is explicit),
\item Imp-A: the filtered Boris method of Algorithm~\ref{alg:boris}  with $\theta$ of \eqref{theta},
\item Two P-A:  the two-point filtered Boris method of Algorithm~\ref{alg:boris-2}.
\end{itemize}
The errors
 in $x$ and $ v_\parallel, v_\perp$ against different
$\epsilon=1/2^j$ are displayed in Fig.~\ref{figure-err1}, where
$j=4,\ldots,13$.   Then we fix $\epsilon=1/2^{10}$ and show the errors at $t=1$ against $h/\epsilon$  in  Fig.~\ref{figure-err11}.
It is observed that all three filtered Boris methods improve considerably over the standard Boris method, and the optimally filtered methods Imp-A and Two P-A show second order, whereas method Exp-A only shows first order. Methods Exp-A and Two P-A behave very similar away from stepsize resonances, but method Two P-A appears more robust near stepsize resonances. For the implicit methods Imp-A and Two P-A, the error behaviour remains essentially unchanged after just one fixed-point iteration.

 \begin{figure}[ptb]
\centering
\includegraphics[width=3.9cm,height=4.8cm]{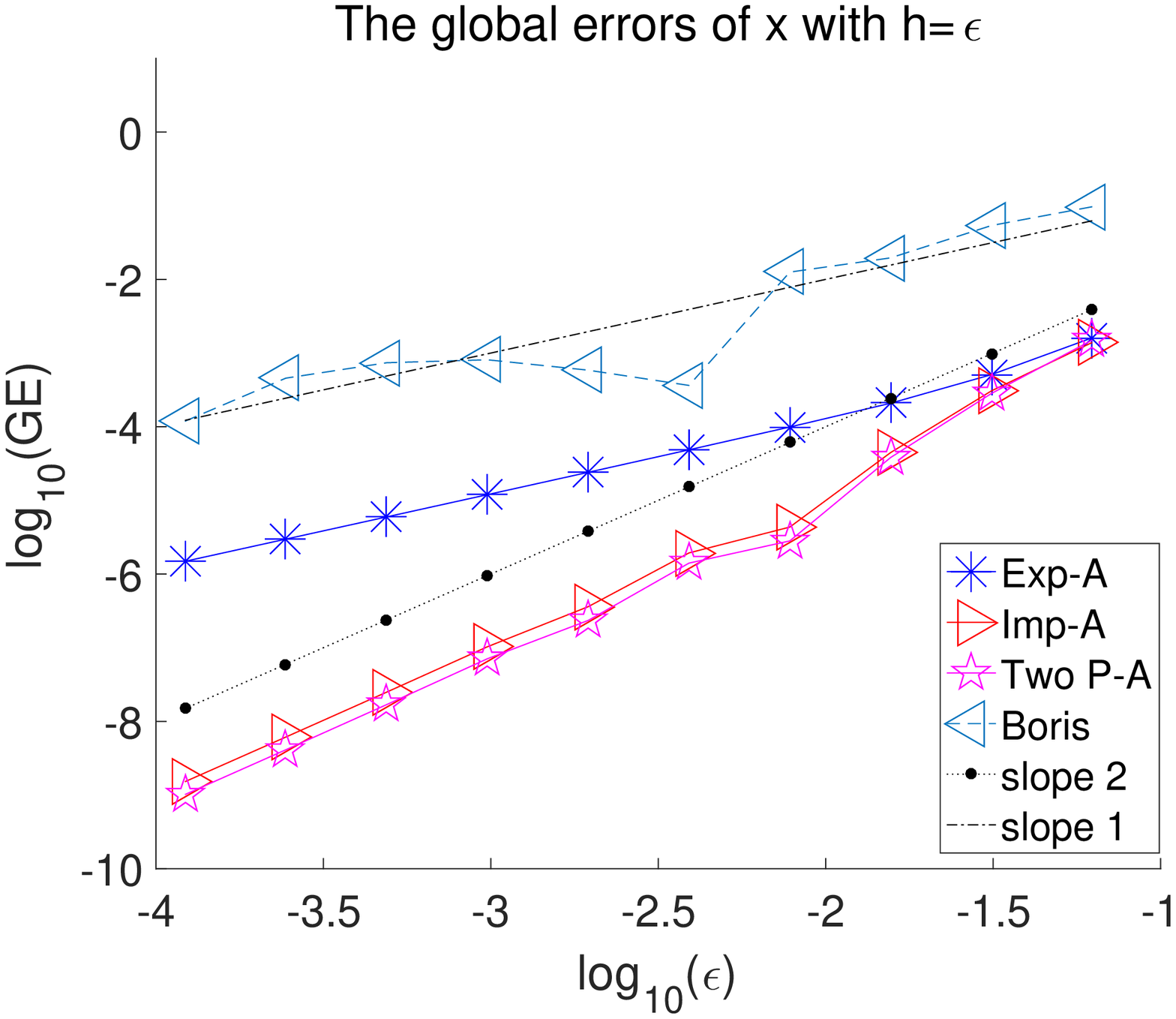}
\includegraphics[width=3.9cm,height=4.8cm]{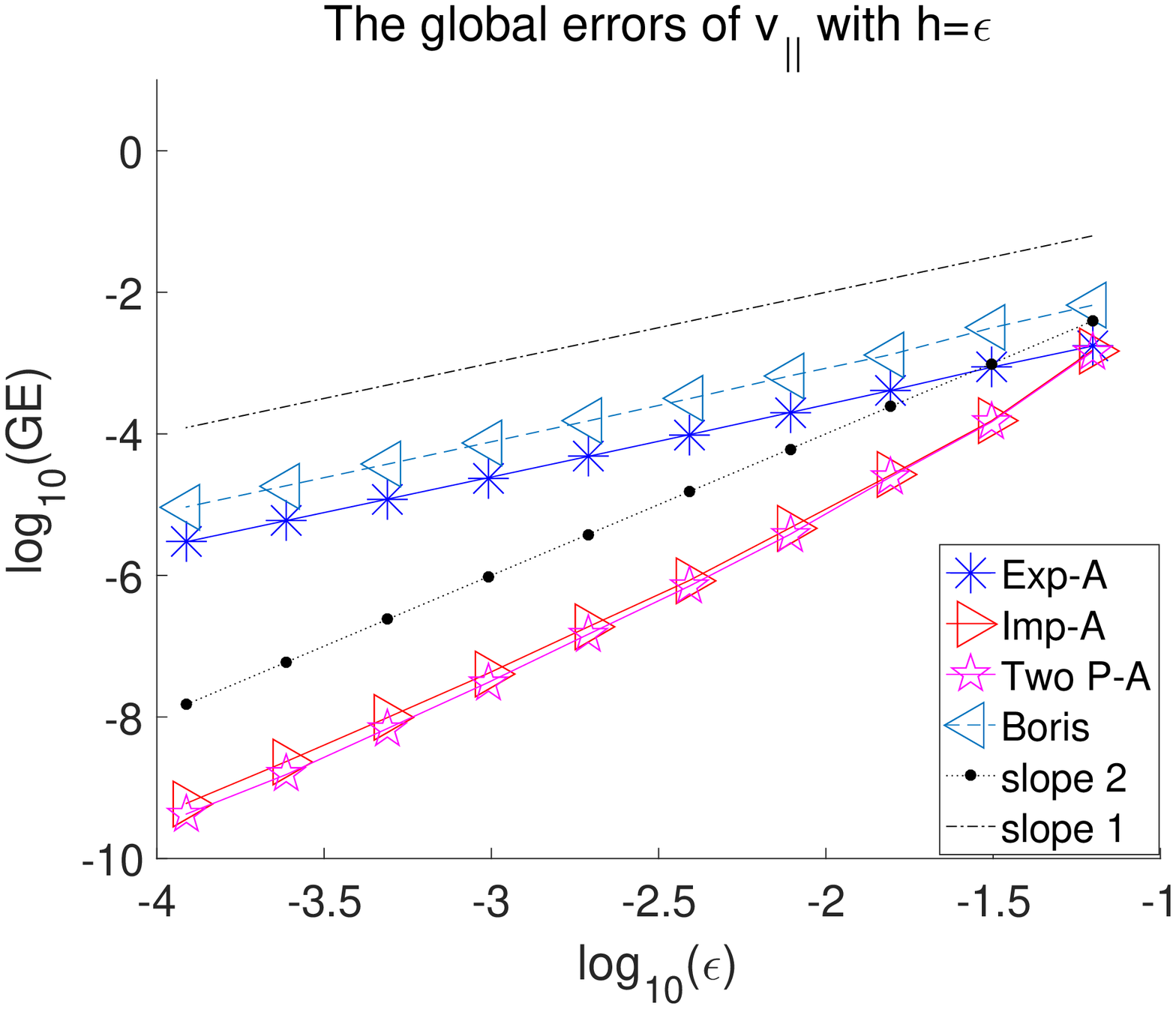}
\includegraphics[width=3.9cm,height=4.8cm]{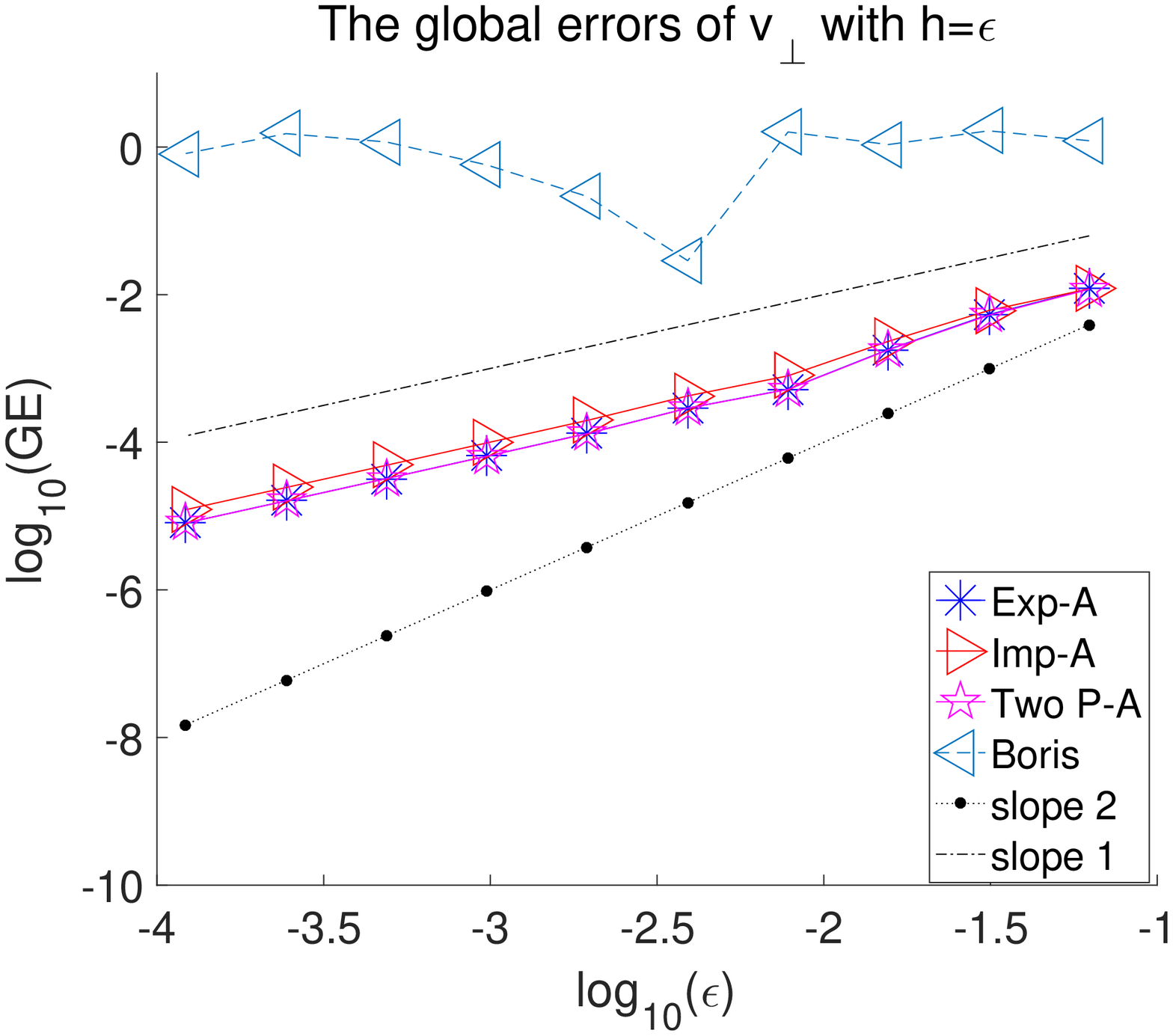}\\
\includegraphics[width=3.9cm,height=4.8cm]{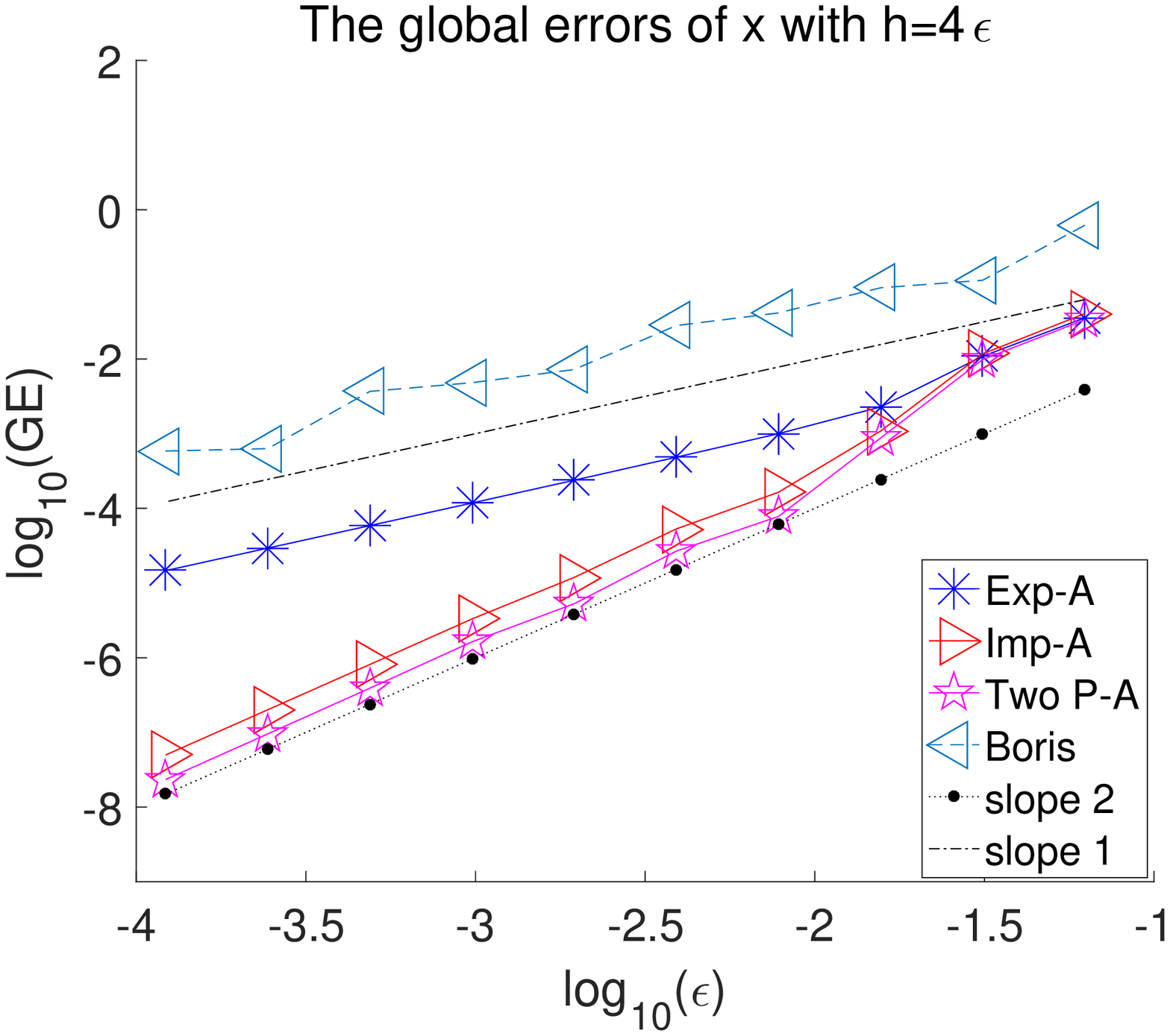}
\includegraphics[width=3.9cm,height=4.8cm]{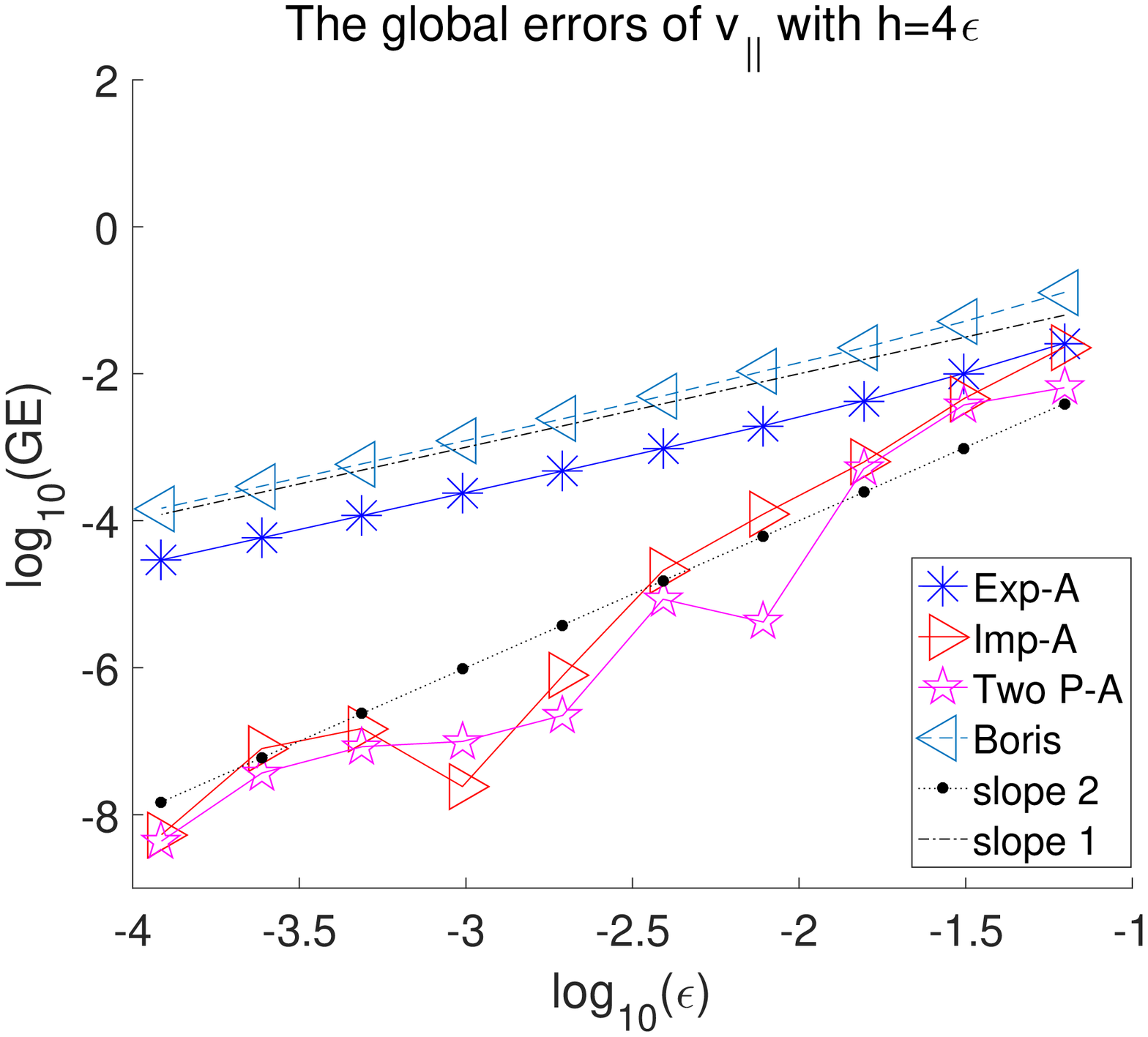}
\includegraphics[width=3.9cm,height=4.8cm]{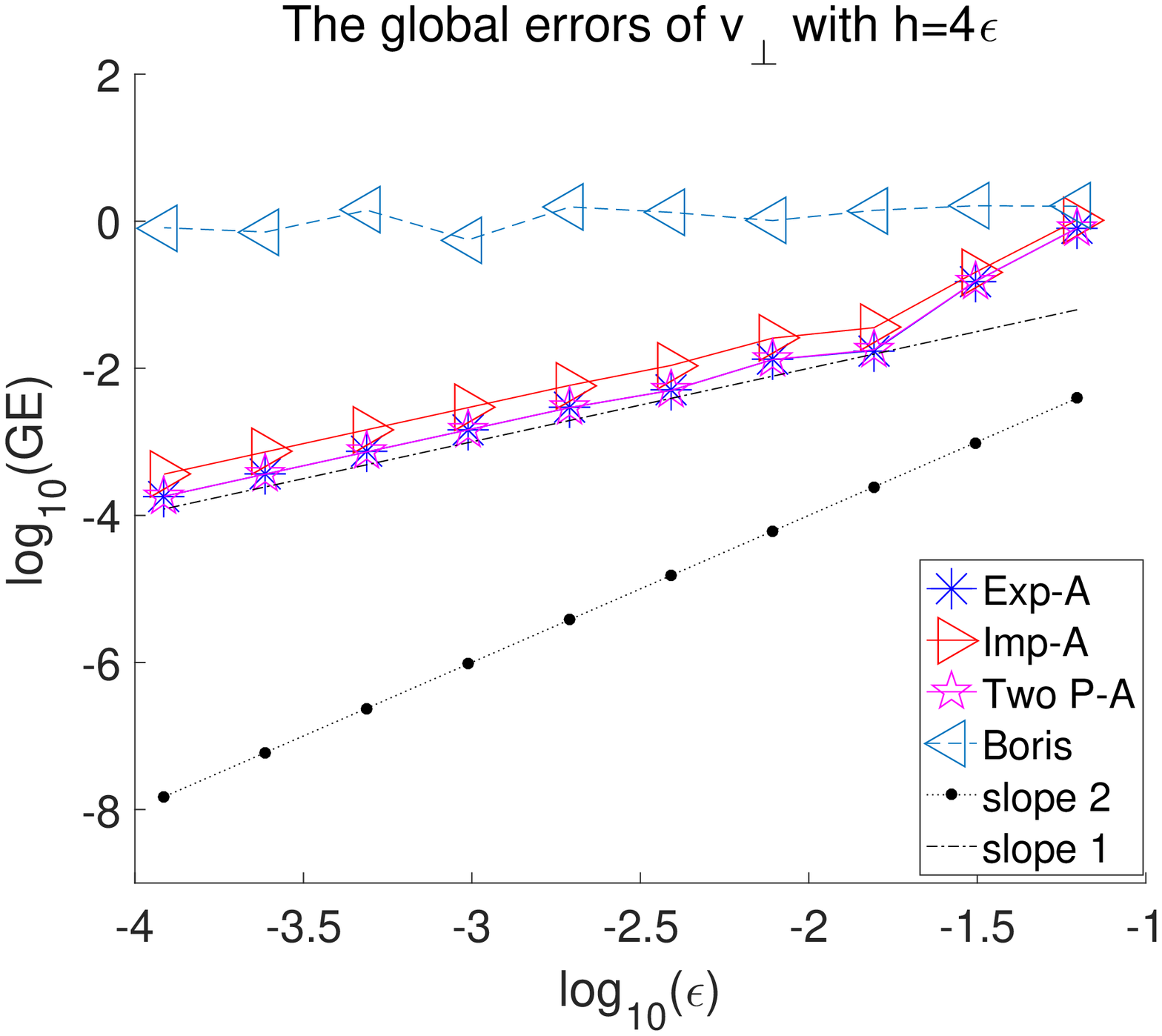}\\
\includegraphics[width=3.9cm,height=4.8cm]{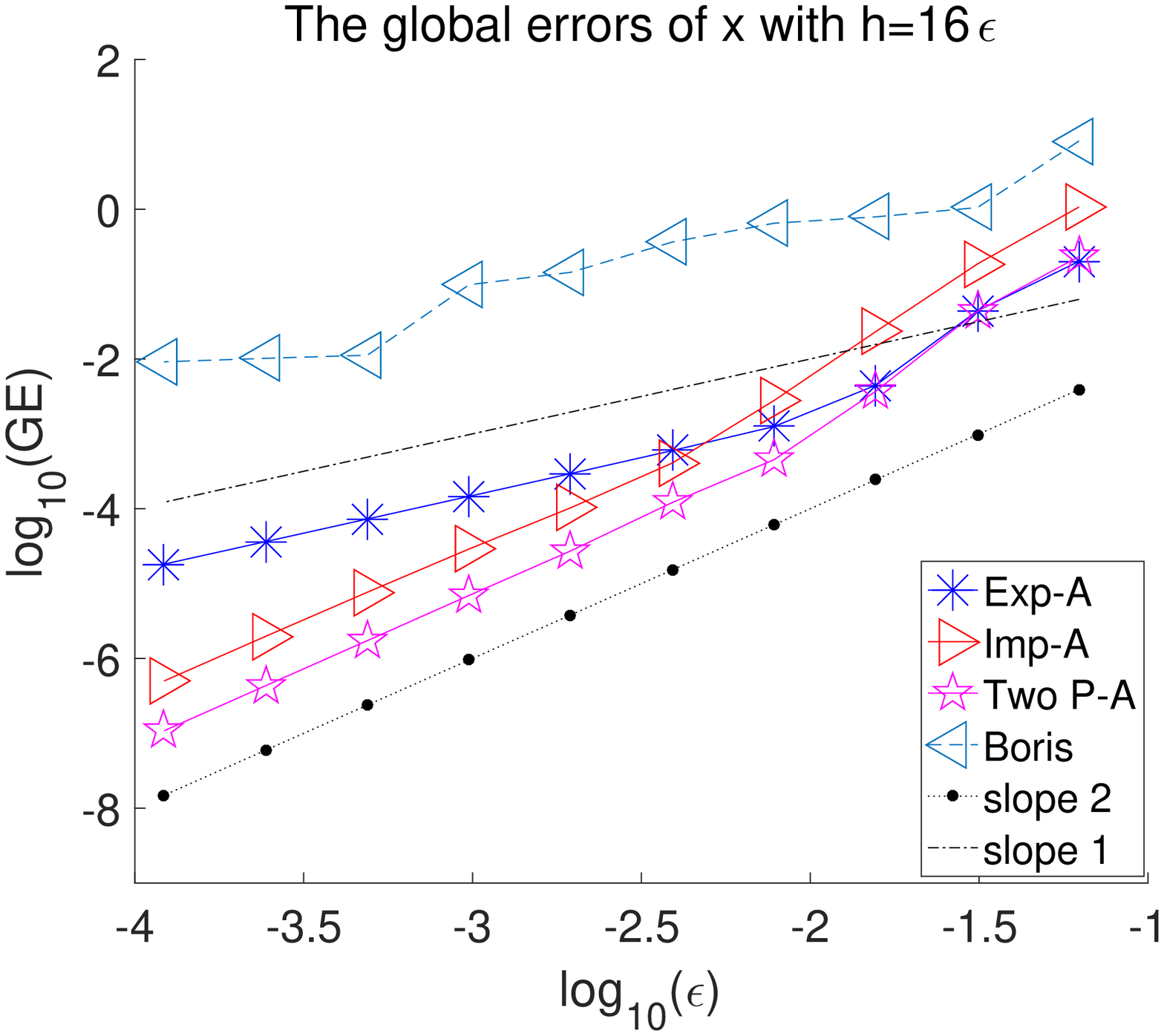}
\includegraphics[width=3.9cm,height=4.8cm]{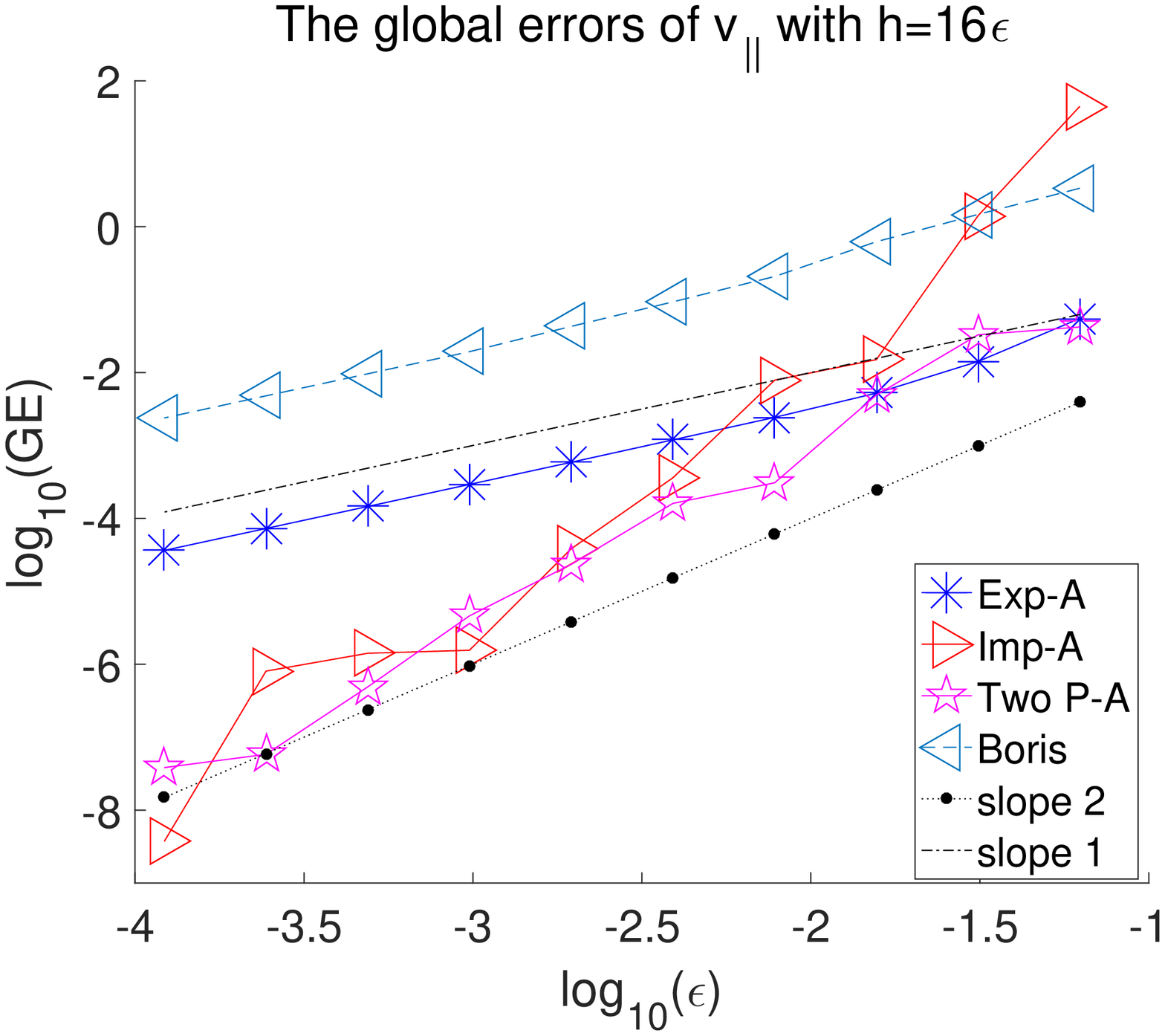}
\includegraphics[width=3.9cm,height=4.8cm]{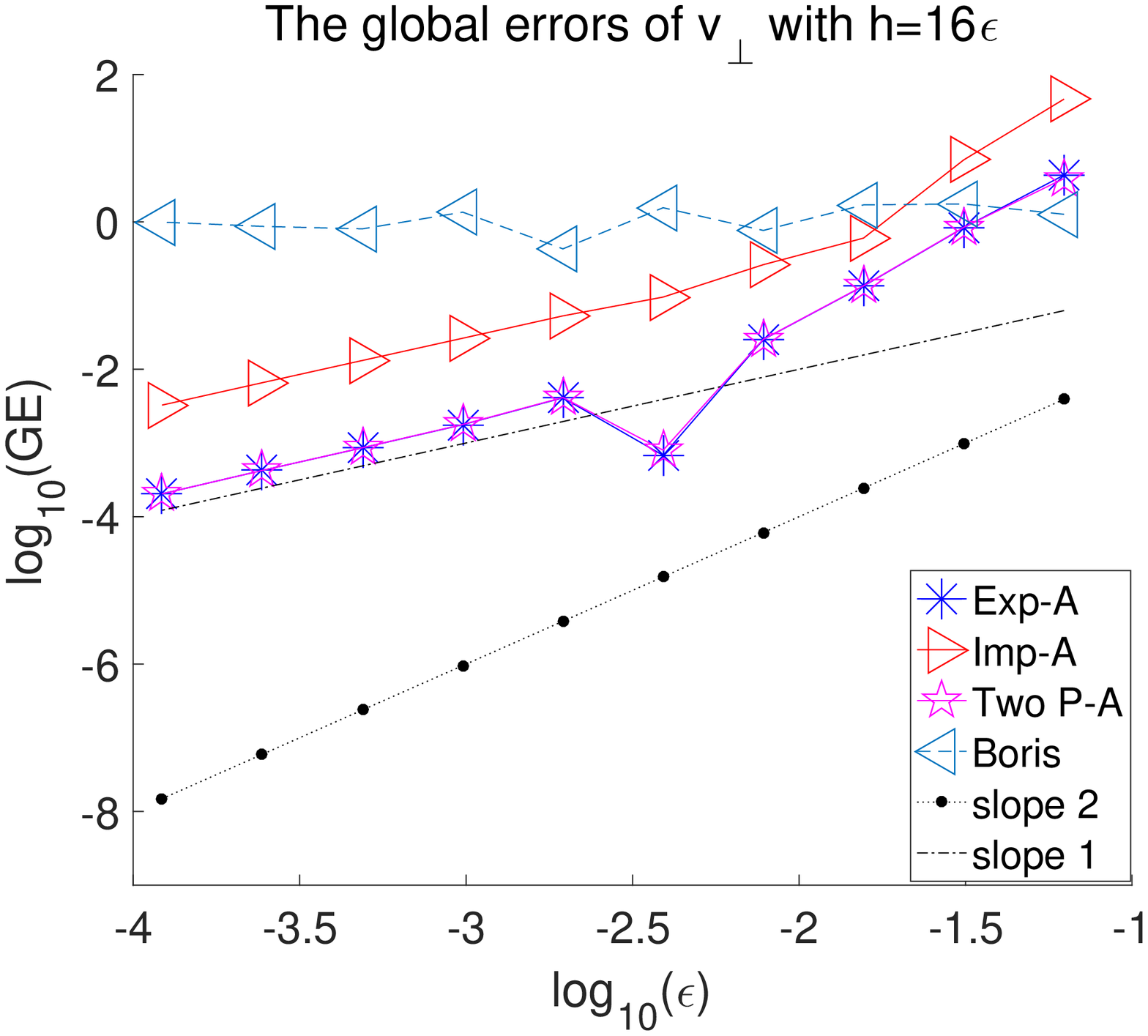}\\
\caption{The logarithm of the  global error against the logarithm
of~$\epsilon$. } \label{figure-err1}
\end{figure}

 \begin{figure}[ptb]
\centering
\includegraphics[width=3.9cm,height=3.8cm]{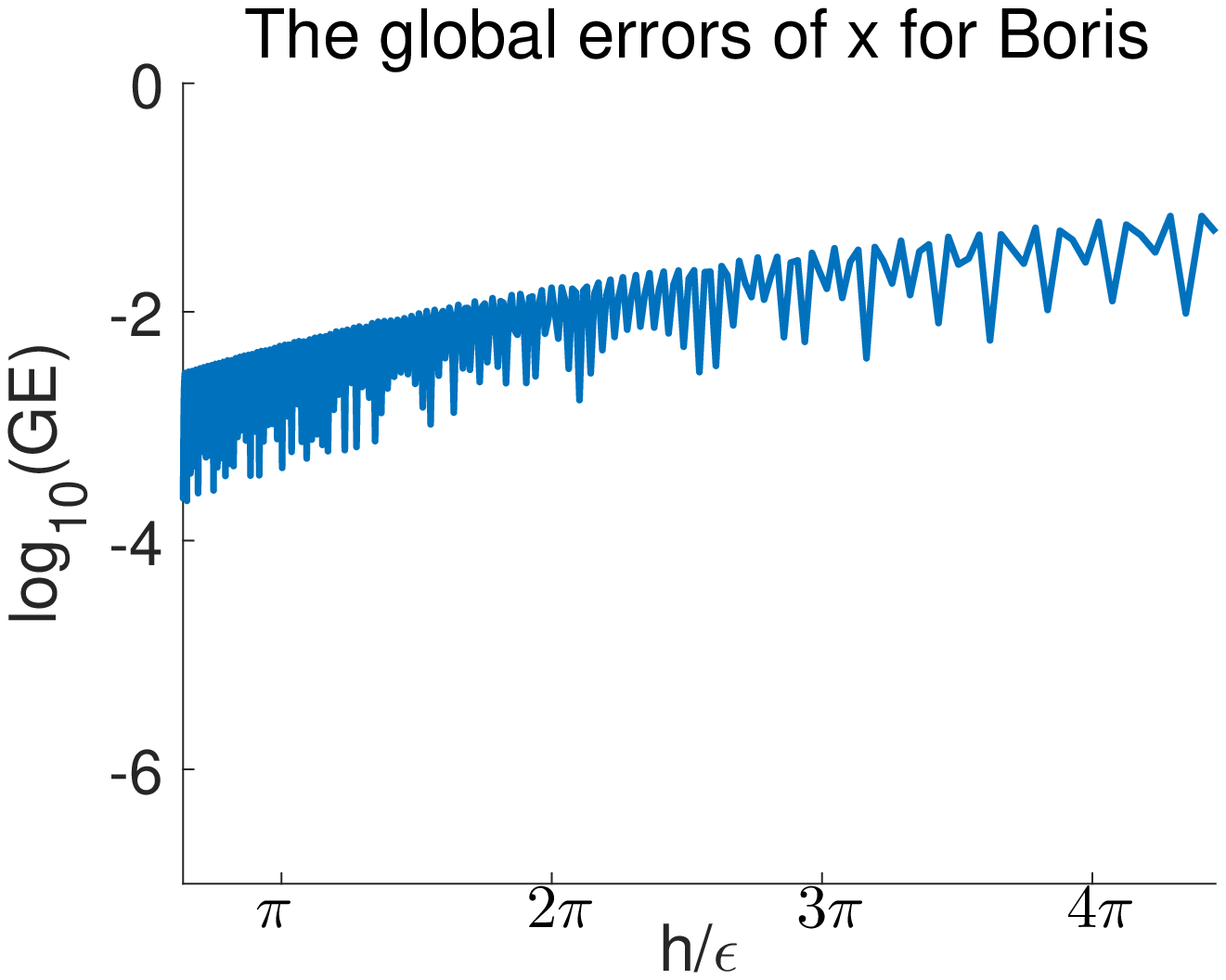}
\includegraphics[width=3.9cm,height=3.8cm]{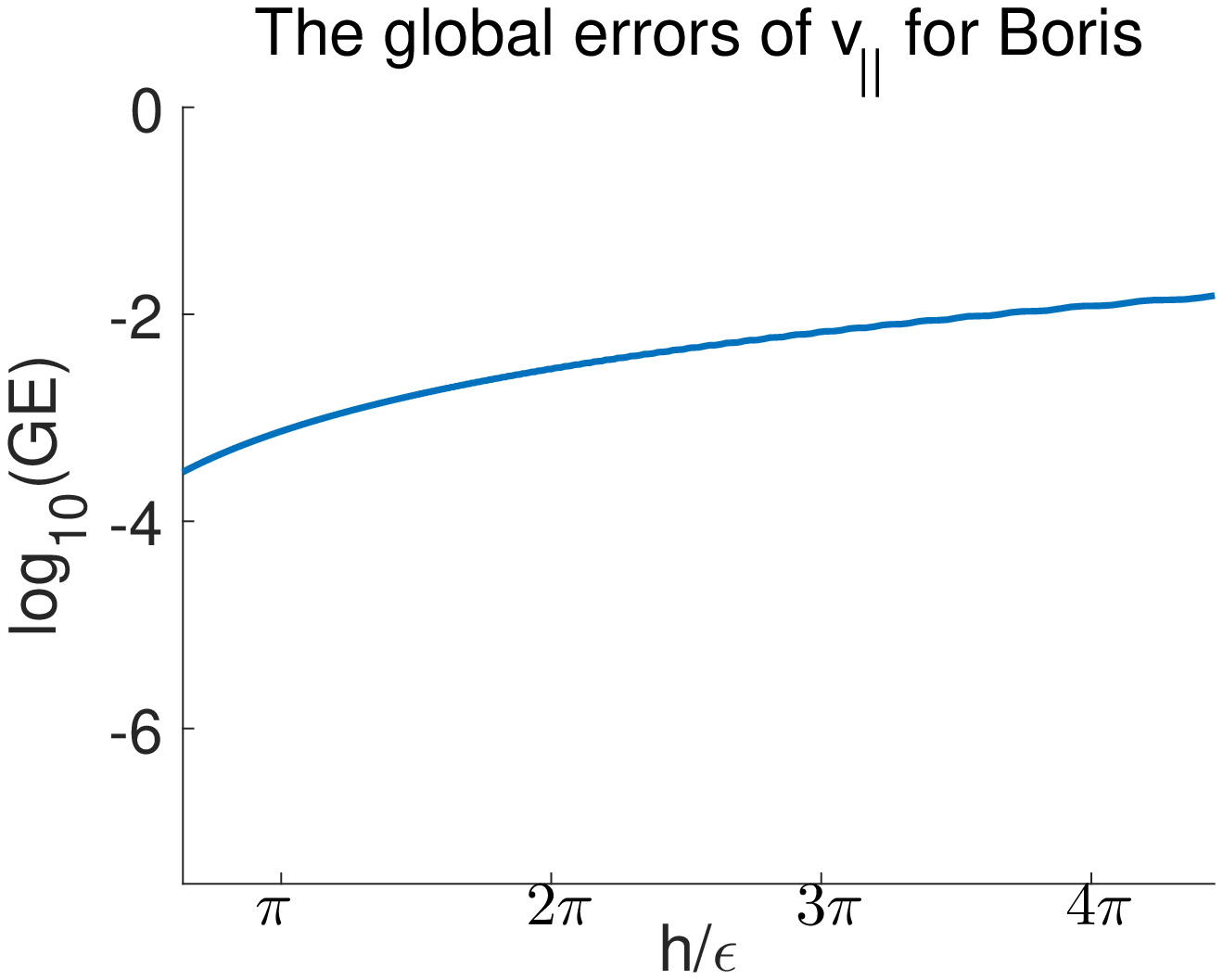}
\includegraphics[width=3.9cm,height=3.8cm]{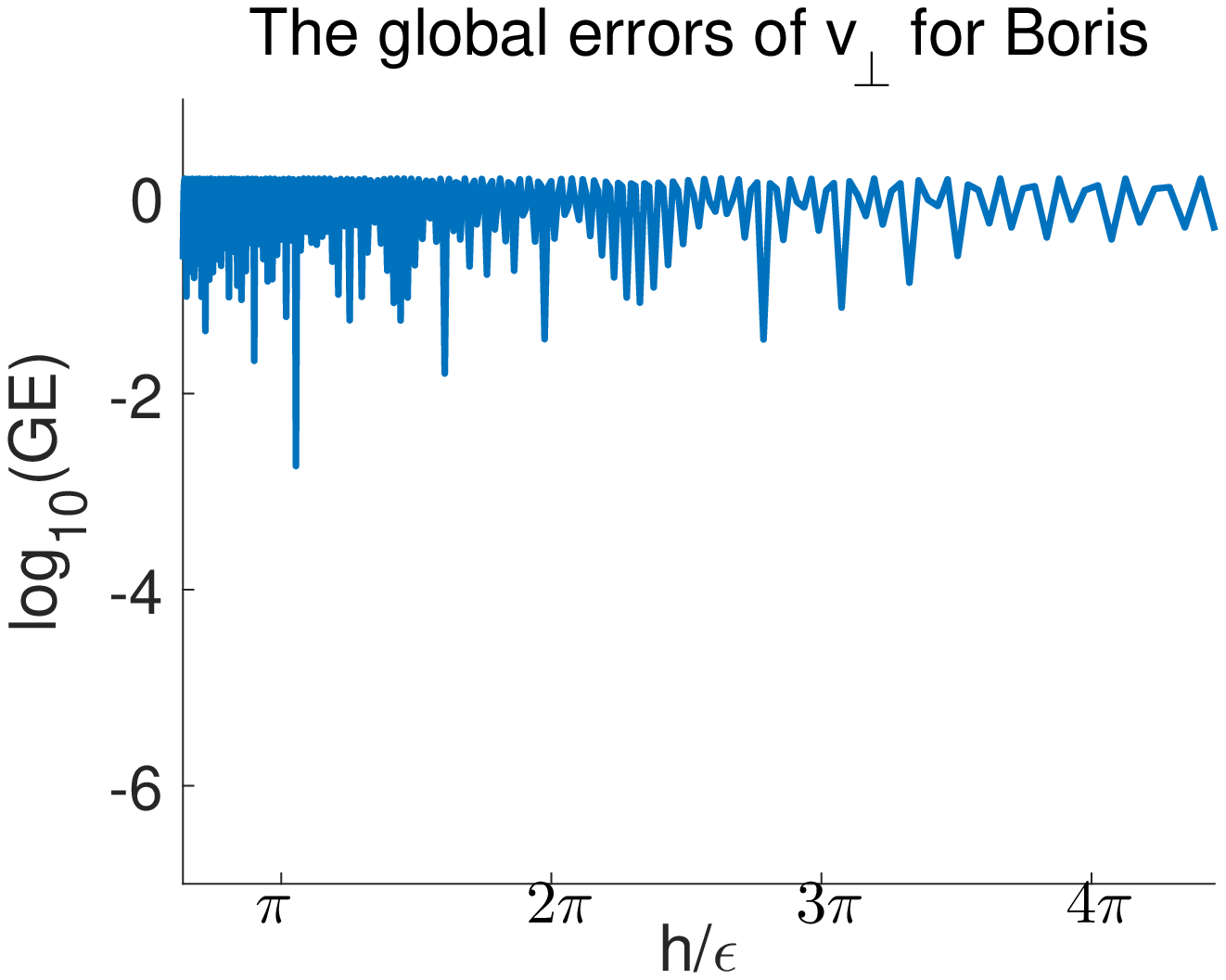}\\
\includegraphics[width=3.9cm,height=3.8cm]{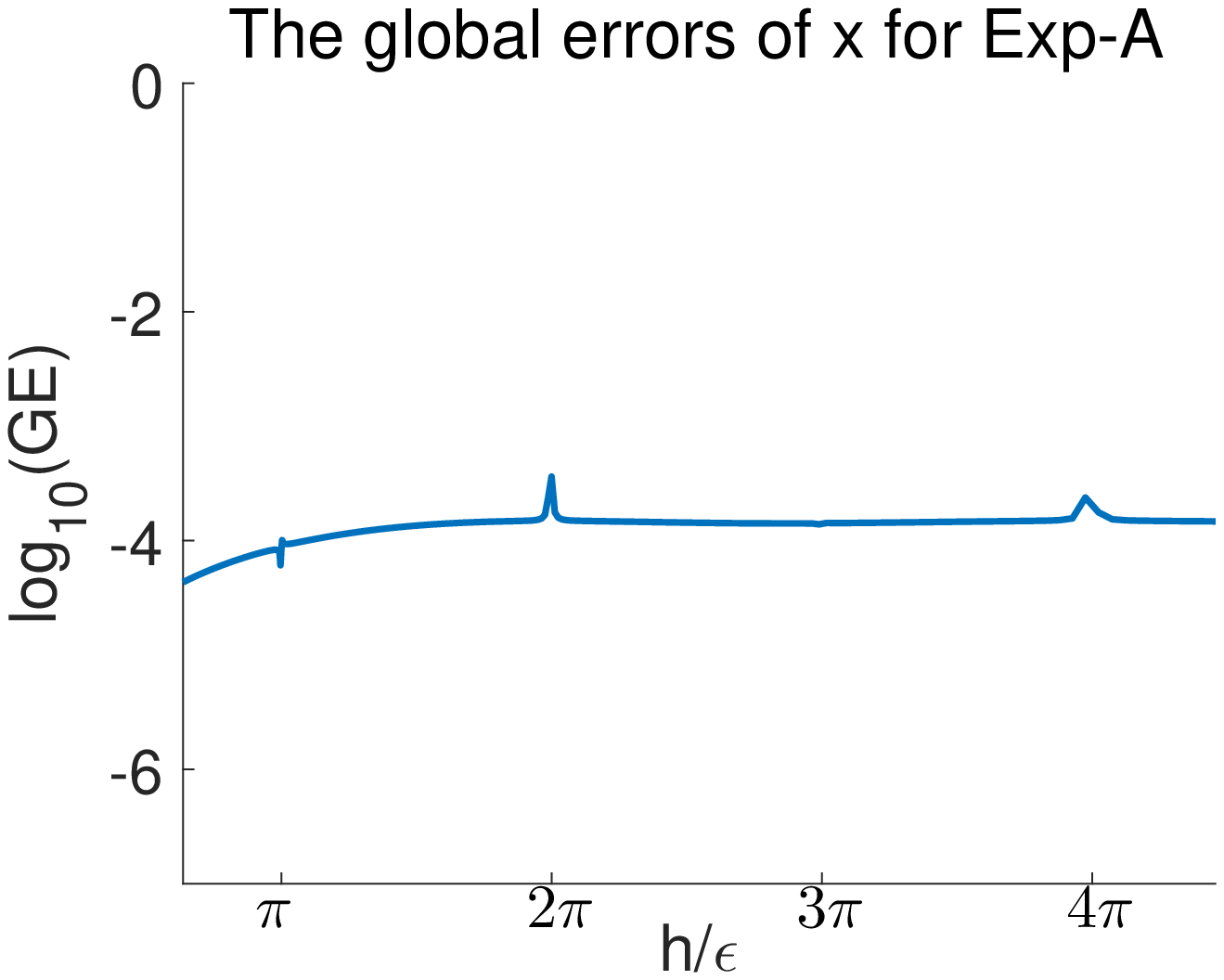}
\includegraphics[width=3.9cm,height=3.8cm]{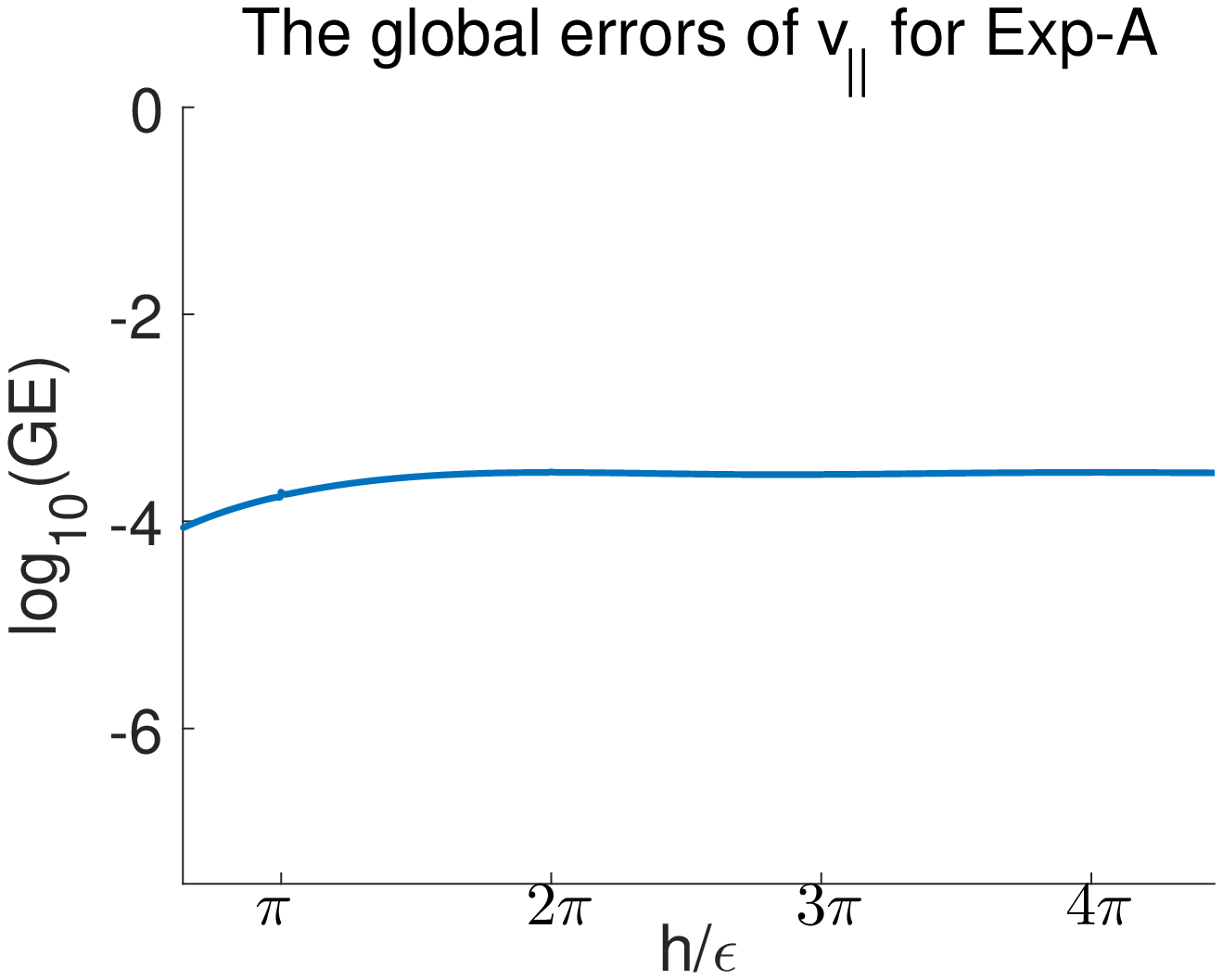}
\includegraphics[width=3.9cm,height=3.8cm]{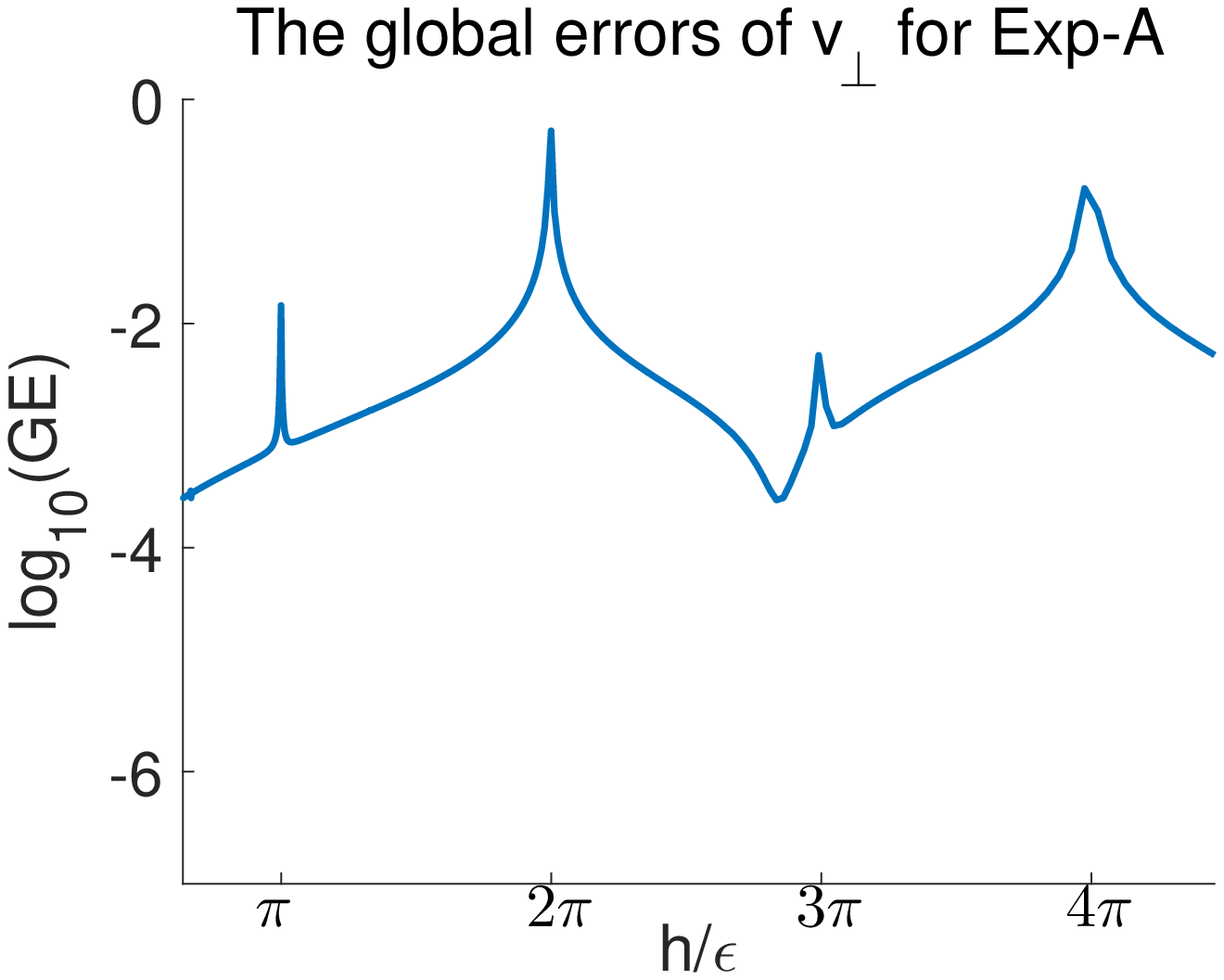}\\
\includegraphics[width=3.9cm,height=3.8cm]{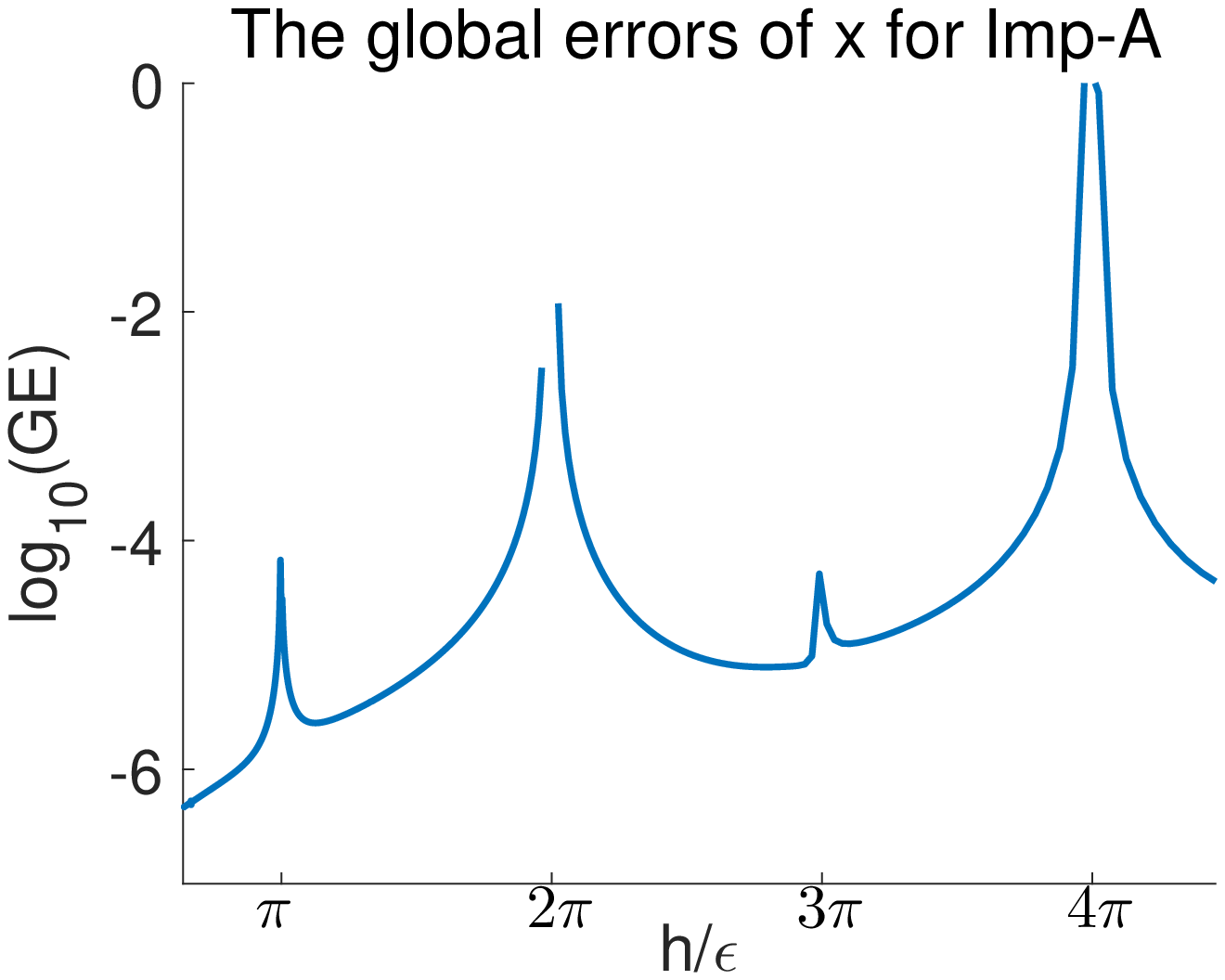}
\includegraphics[width=3.9cm,height=3.8cm]{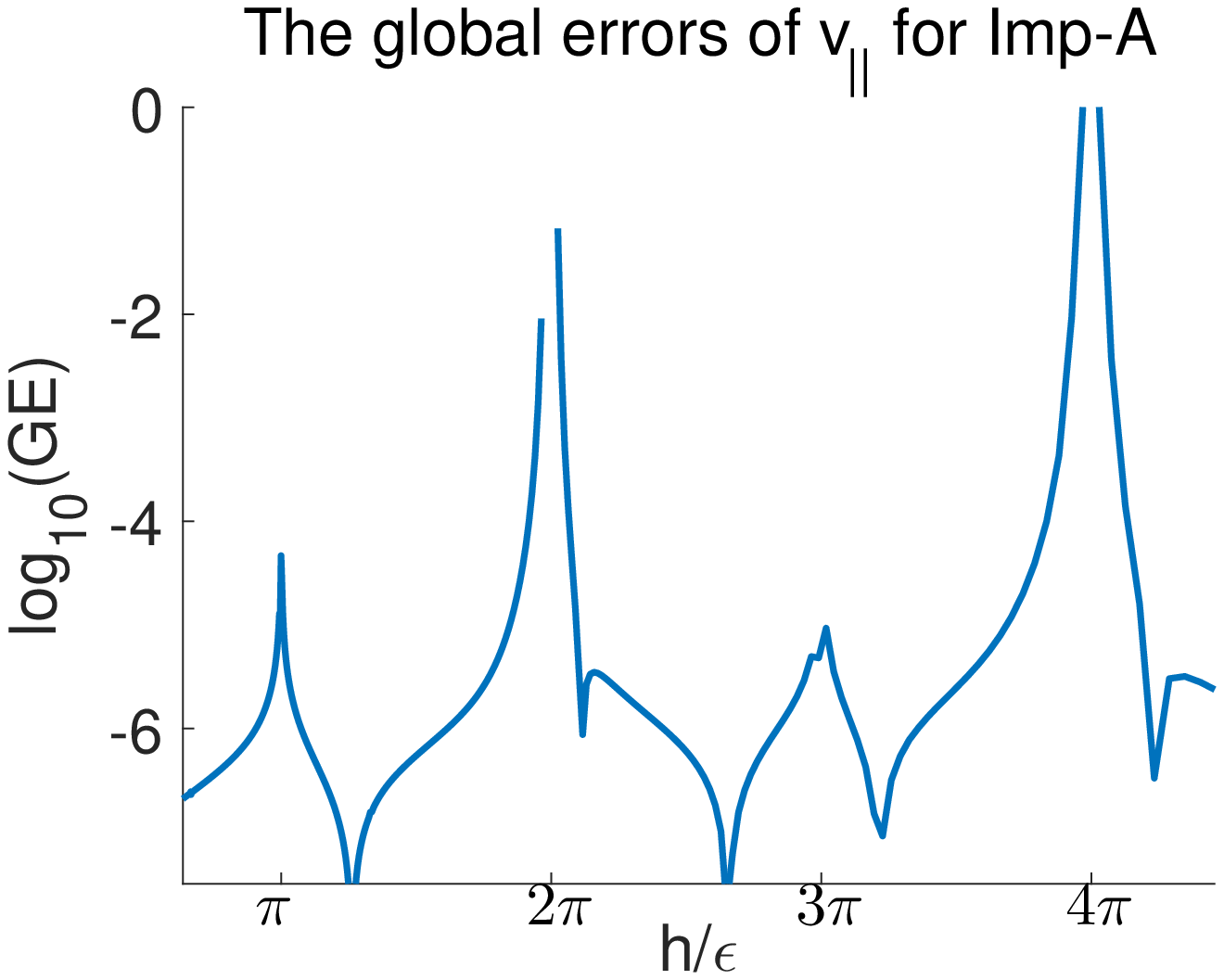}
\includegraphics[width=3.9cm,height=3.8cm]{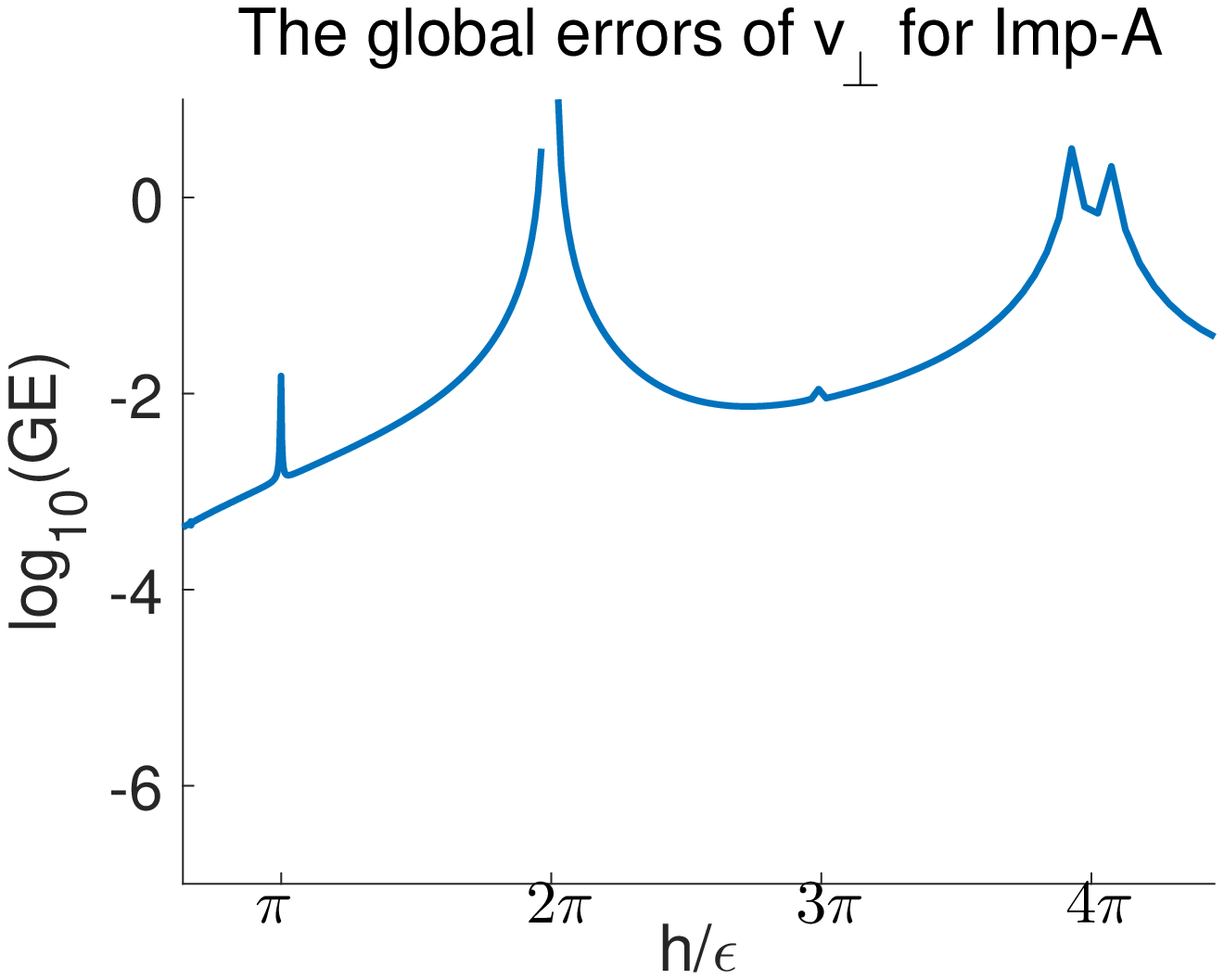}\\
\includegraphics[width=3.9cm,height=3.8cm]{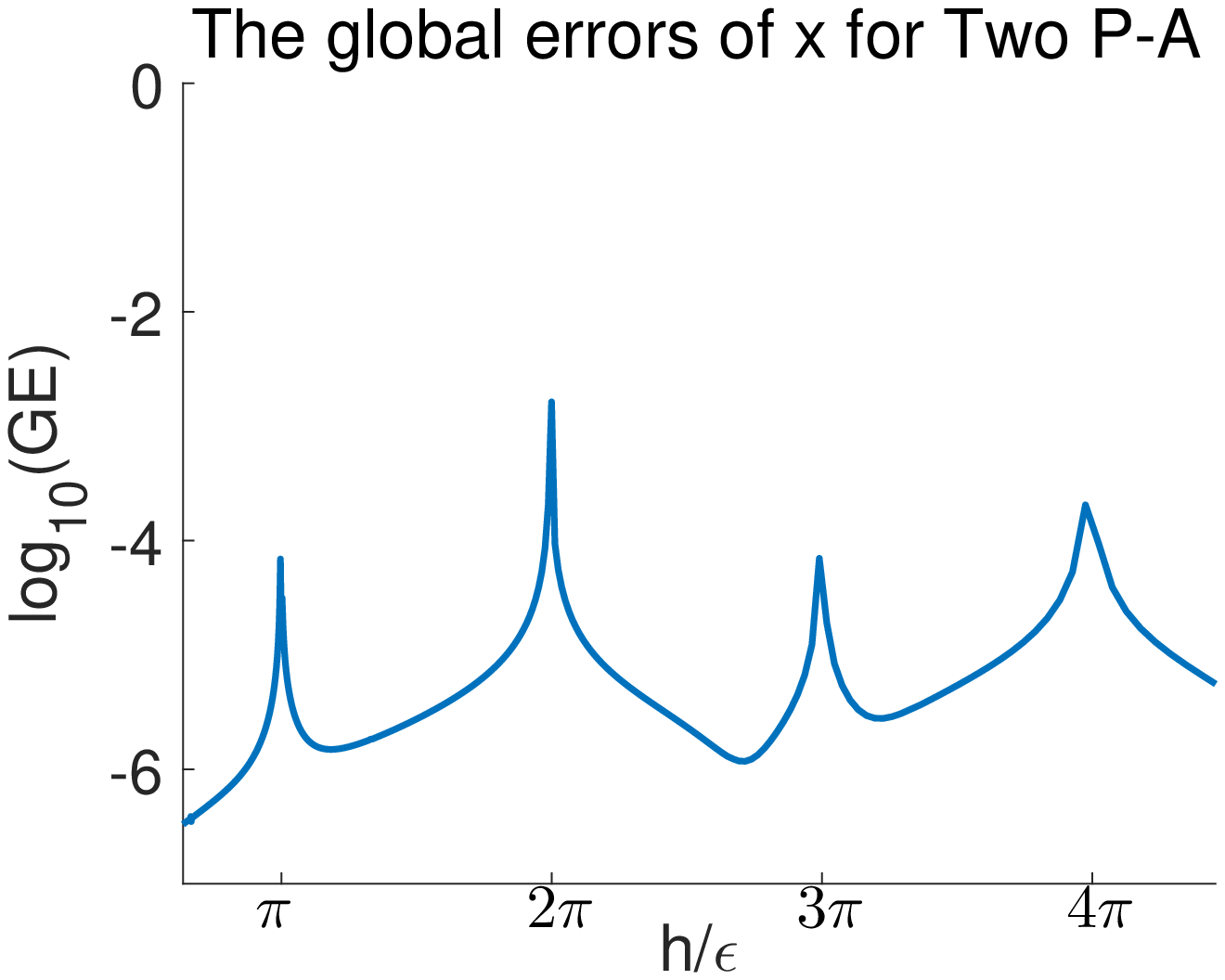}
\includegraphics[width=3.9cm,height=3.8cm]{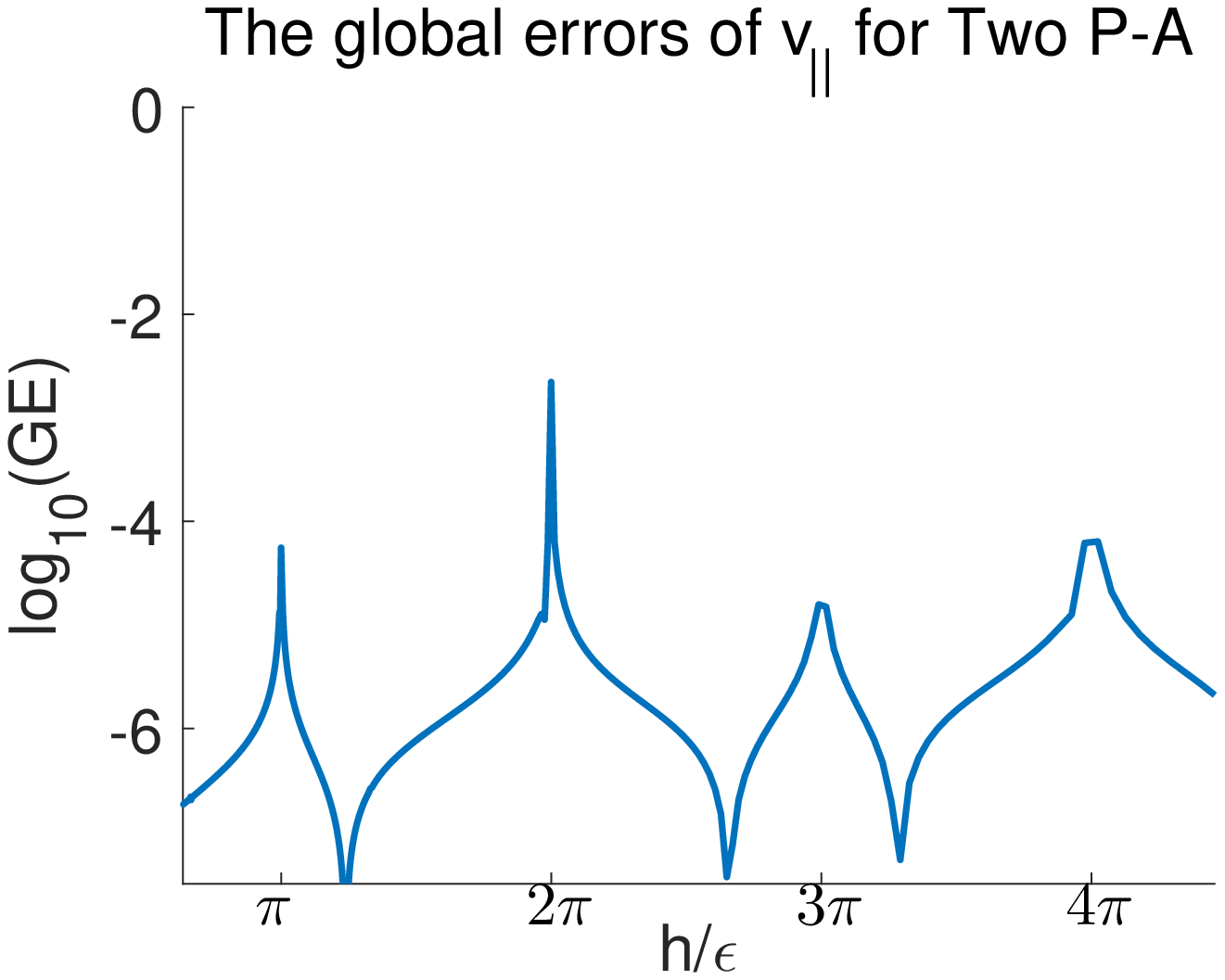}
\includegraphics[width=3.9cm,height=3.8cm]{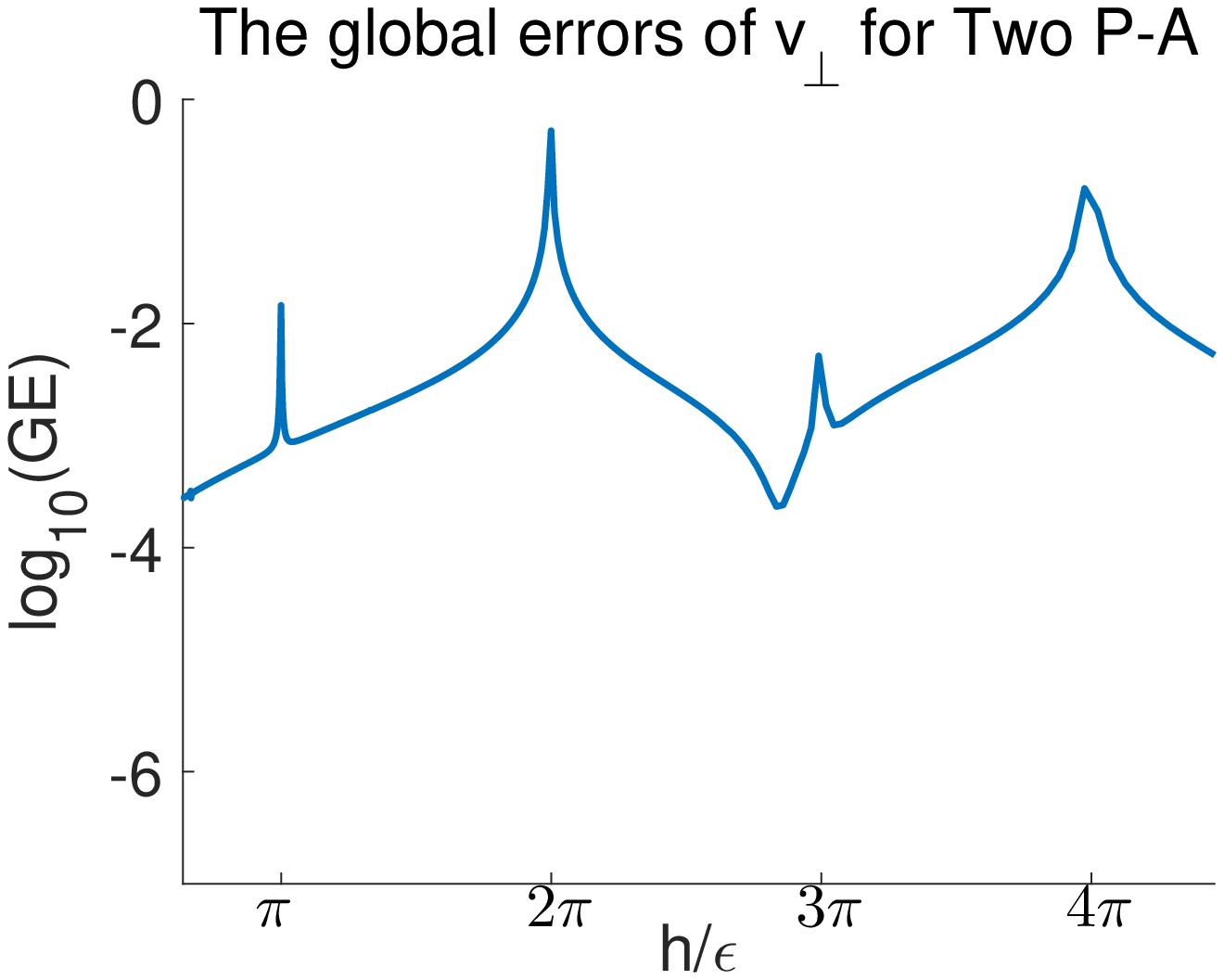}
\caption{ The logarithm of the  global error at  $t = 1$ against
 $h/\epsilon$  for $\epsilon=1/2^{10}$ and $h=1/k$, where $k=60,61,\ldots,600$.}
\label{figure-err11}
\end{figure}

\section*{Appendix: Implementation}

The filtered Boris algorithm requires the computation of matrix functions applied to
a vector. This can be done very efficiently with a Rodriguez-like formula.
Consider a vector $B=(b_1,b_2,b_3)^\top\in \real^3$ and the skew-symmetric
matrix $\widehat B$ of \eqref{skew-hat}, and let $b= |B|$. Assume that the
function $\varphi (\zeta)$ can be expanded into a Taylor series at the origin with real coefficients $c_n$,
and write
\[
\varphi (\iu y ) = \varphi (0) + \iu y\varphi_1 (y) - y^2\varphi_2 (y)
\]
with $\varphi_1(y) = \sum_{j\ge 0} c_{2j+1} (-y^2)^j$ and $\varphi_2(y) = \sum_{j\ge 0} c_{2j+2} (-y^2)^j$.
The fact that
\[
\widehat B^3 = -b^2 \widehat B
\]
implies that
\begin{equation}\label{rodriguez}
\varphi (\widehat B ) = \varphi (0) I + \varphi_1 (b) \widehat B + \varphi_2 (b) \widehat B^2 .
\end{equation}
This permits us to compute $\varphi (\widehat B) v$ by evaluating the scalars
$\varphi (0), \varphi_1 (b), \varphi_2 (b)$, and by forming twice a product of
$\widehat B$ with a vector. Note that $\widehat B v = B\times v $.

For the case that $\varphi (\zeta) $ has only even powers of $\zeta$, we have $\varphi_1 (y)=0$,
and the formula simplifies. Similarly, for the case where only odd powers of $\zeta$
are present, we have $\varphi (0)=0$ and  $\varphi_2 (y)=0$.
For the matrix functions of Algorithm~\ref{alg:boris} we thus have
\begin{align*}
\exp (-h\widehat B) & = I - \frac{\sin (hb)}{b} \widehat B + \frac{1-\cos (hb)}{b^2} \widehat B^2 ,\\
\Psi(h\widehat B) & = I + \frac{1 - \tanc (hb/2)}{b^2} \widehat B^2 ,\\
\Phi_1 (h\widehat B) & = I + \frac{1- \sinc (hb)^{-1}}{b^2} \widehat B^2 ,\\
\Upsilon (h\widehat B) & = \frac{1-\sinc(hb)^{-1} }{hb^2} \widehat B .
\end{align*}

\section*{Acknowledgement} C.L. thanks Eric Sonnendr\"ucker for stimulating discussions during the Oberwolfach workshop 2019-04.
This work was supported by the Fonds National Suisse, Project No.\ 200020-159856, 
by Deutsche For\-schungsgemeinschaft, SFB 1173, and by the Humboldt Foundation.

\renewcommand{\refname}{\normalsize\bf References}
\small
\bibliographystyle{acm}
\bibliography{HLW}

\begin{thebibliography}{10}

\bibitem{boris70rps}
{\sc Boris, J.~P.}
\newblock Relativistic plasma simulation-optimization of a hybrid code.
\newblock {\em Proceeding of Fourth Conference on Numerical Simulations of
  Plasmas\/} (November 1970), 3--67.

\bibitem{brizard07fon}
{\sc Brizard, A.~J., and Hahm, T.~S.}
\newblock Foundations of nonlinear gyrokinetic theory.
\newblock {\em Rev. Modern Phys. 79}, 2 (2007), 421--468.

\bibitem{garcia-archilla99lmf}
{\sc Garc{\'\i}a-Archilla, B., Sanz-Serna, J.~M., and Skeel, R.~D.}
\newblock Long-time-step methods for oscillatory differential equations.
\newblock {\em SIAM J.\ Sci.\ Comput. 20\/} (1999), 930--963.

\bibitem{hairer16lta}
{\sc Hairer, E., and Lubich, C.}
\newblock Long-term analysis of the {S}t\"ormer-{V}erlet method for
  {H}amiltonian systems with a solution-dependent high frequency.
\newblock {\em Numer.\ Math. 34\/} (2016), 119--138.

\bibitem{hairer18ebo}
{\sc Hairer, E., and Lubich, C.}
\newblock Energy behaviour of the {B}oris method for charged-particle dynamics.
\newblock {\em BIT 58\/} (2018), 969--979.

\bibitem{hairer19lta}
{\sc Hairer, E., and Lubich, C.}
\newblock Long-term analysis of a variational integrator for charged-particle
  dynamics in a strong magnetic field.
\newblock {\em Numerische Mathematik, to appear\/} (2019).

\bibitem{hairer06gni}
{\sc Hairer, E., Lubich, C., and Wanner, G.}
\newblock {\em Geometric Numerical Integration. {S}tructure-Preserving
  Algorithms for Ordinary Differential Equations}, 2nd~ed.
\newblock Springer Series in Computational Mathematics 31. Springer-Verlag,
  Berlin, 2006.

\bibitem{hochbruck99agm}
{\sc Hochbruck, M., and Lubich, C.}
\newblock A {G}autschi-type method for oscillatory second-order differential
  equations.
\newblock {\em Numer.\ Math. 83\/} (1999), 403--426.

\bibitem{kruskal58tgo}
{\sc Kruskal, M.}
\newblock The gyration of a charged particle.
\newblock {\em Rept. PM-S-33 (NYO-7903), Princeton University, Project
  Matterhorn\/} (1958).

\bibitem{northrop63tam}
{\sc Northrop, T.~G.}
\newblock {\em The adiabatic motion of charged particles}.
\newblock Interscience Tracts on Physics and Astronomy, Vol. 21. Interscience
  Publishers John Wiley \& Sons\, New York-London-Sydney, 1963.

\bibitem{possanner18gfv}
{\sc Possanner, S.}
\newblock Gyrokinetics from variational averaging: existence and error bounds.
\newblock {\em J. Math. Phys. 59}, 8 (2018), 082702, 34.

\end{thebibliography}

\end{document}